\documentclass[12pt,reqno]{amsart}

\usepackage{amsmath,amssymb,ifthen}
\usepackage{fullpage}
\usepackage{graphicx,psfrag,subfigure}
\usepackage{color}
\usepackage{moreverb}
\usepackage{amsthm}
\usepackage{mathrsfs}
\usepackage{hhline}
\usepackage{bbm}
\usepackage[english]{babel}
\usepackage{tikz}
\usepackage[utf8]{inputenc}
\usepackage{amsfonts}
\usepackage{enumerate}

\def\MM{\mathcal M}

\def\D{D}
\def\T{\mathbb{T}}
\def\C{{\mathbb C}}
\def\R{{\mathbb R}}

\def\N{{\mathbb N}}

\def\OO{{\mathcal O}}
\def\PP{{\mathcal P}}

\def\RR{\mathcal{R}} 
\def\SS{{\mathcal S}}
\def\TT{{\mathcal T}}
\def\XX{{\mathcal X}}

\def\diam{{\rm diam}}
\def\q#1{q_{\rm #1}}

\def\ro#1{\rho_{\rm #1}}

\def\const#1{C_{\rm #1}}

\def\norm#1#2{\|#1\|_{#2}}
\def\seminorm#1#2{\vert #1\vert_{#2}}
\def\set#1#2{\big\{#1\,:\,#2\big\}}
\def\sprod#1#2{(#1\,;\,#2)}
\def\dual#1#2{\langle#1\,;\,#2\rangle}

\def\refine{{\tt refine}}
\def\supp{{\rm supp}}

\def\dist{{\rm dist}}

\def\coarse{\bullet}
\def\fine{\circ}

\numberwithin{equation}{section}
\numberwithin{figure}{section}
\newtheorem{theorem}{Theorem}[section]
\newtheorem{proposition}[theorem]{Proposition}
\newtheorem{lemma}[theorem]{Lemma}
\newtheorem{corollary}[theorem]{Corollary}
\newtheorem{algorithm}[theorem]{Algorithm}

\newtheorem{remark}[theorem]{Remark}

\def\subsection#1
{
 \bigskip

 \refstepcounter{subsection}
 {\noindent\bf\arabic{section}.\arabic{subsection}.~#1.~}
}
\renewcommand{\subsection}[1]{\refstepcounter{subsection}\medskip{\bf\thesubsection.~#1.}}

\newcommand*\patchAmsMathEnvironmentForLineno[1]{%
  \expandafter\let\csname old#1\expandafter\endcsname\csname #1\endcsname
  \expandafter\let\csname oldend#1\expandafter\endcsname\csname end#1\endcsname
  \renewenvironment{#1}%
     {\linenomath\csname old#1\endcsname}%
     {\csname oldend#1\endcsname\endlinenomath}}%
\newcommand*\patchBothAmsMathEnvironmentsForLineno[1]{%
  \patchAmsMathEnvironmentForLineno{#1}%
  \patchAmsMathEnvironmentForLineno{#1*}}%
\AtBeginDocument{%
\patchBothAmsMathEnvironmentsForLineno{equation}%
\patchBothAmsMathEnvironmentsForLineno{align}%
\patchBothAmsMathEnvironmentsForLineno{flalign}%
\patchBothAmsMathEnvironmentsForLineno{alignat}%
\patchBothAmsMathEnvironmentsForLineno{gather}%
\patchBothAmsMathEnvironmentsForLineno{multline}%
}
\usepackage[mathlines]{lineno}

\title{Adaptive BEM for elliptic PDE systems, \\
Part I: Abstract framework \\
for weakly-singular integral equations}

\author{Gregor Gantner}
\address{Korteweg-de Vries Institute for Mathematics\\
University of Amsterdam, P.O. Box 94248, 1090 GE Amsterdam, The Netherlands}
\email{G.Gantner@uva.nl}

\author{Dirk Praetorius}
\address{Institute for Analysis and Scientific Computing\\
TU Wien, Wiedner Hauptstra\ss{}e 8-10, A-1040 Wien, Austria}
\email{Dirk.Praetorius@asc.tuwien.ac.at}

\begin{document}

\begin{abstract}
In the present work, we consider weakly-singular integral equations arising from linear second-order strongly-elliptic PDE systems with constant coefficients, including, e.g., linear elasticity.  
We introduce a general framework for optimal convergence of adaptive Galerkin BEM.  
We identify certain abstract properties for the underlying meshes, the corresponding mesh-refinement strategy, and the ansatz spaces that guarantee convergence at optimal algebraic rate of an  adaptive algorithm driven by the weighted-residual error.
These properties are satisfied, e.g.,  for discontinuous piecewise polynomials on simplicial meshes as well as certain ansatz spaces used for isogeometric analysis.
Technical contributions include local inverse estimates for the (non-local) boundary integral operators associated to the PDE system.
\end{abstract}

\date{\today}
\maketitle

\section{Introduction}

\subsection{State of the art}
For the Laplace model problem, adaptive boundary element methods (BEM) using (dis)continuous piecewise polynomials on triangulations have been intensively studied in the literature. 
In particular, optimal convergence of mesh-refining adaptive algorithms has been proved for polyhedral boundaries~\cite{part1,part2,fkmp} as well as smooth boundaries~\cite{gantumur}. The  work~\cite{invest} allows to transfer these results to piecewise smooth boundaries; see also the discussion in the review article~\cite{axioms}.
In \cite{helmholtzbem}, these results have been generalized  to the Helmholtz problem. 
In recent years, we have also shown optimal convergence of adaptive isogeometric BEM (IGABEM) using one-dimensional splines for the 2D Laplace problem \cite{resigaconv,hypiga}. However, the important case of 3D IGABEM remained open.
Moreover, a generalization to other PDE operators is highly non-trivial (see~\eqref{eq:one invest} below), but especially linear elasticity is of great interest in the context of isogeometric analysis. 

In \cite{igafem}, we have considered isogeometric finite element methods (IGAFEM). 
We have derived an abstract framework which guarantees that, first, the classical residual FEM error estimator is reliable, and second, the related adaptive algorithm yields optimal convergence; see \cite[Section~2 and~4]{igafem}.
We then showed that, besides standard FEM with piecewise polynomials, this abstract framework covers IGAFEM with hierarchical splines (see~\cite[Section~3 and~5]{igafem}) as well as IGAFEM with analysis suitable T-splines (see the recent work~\cite{tigafem}).

The aim of the present work is to develop such an abstract framework also for BEM, which is mathematically much more demanding than FEM.
In ongoing research~\cite{riga}, we aim to show that this framework covers, besides standard discretizations with piecewise polynomials, also IGABEM with hierarchical splines resp.\ T-splines.

To this end, the present work focusses on weakly-singular integral equations.
For a given Lipschitz domain $\Omega\subseteq\R^d$ with compact boundary $\Gamma:=\partial\Omega$ and right-hand side $f:\Gamma\to\C$, we consider 
\begin{align}
({\mathfrak{V}}\phi)(x):=\int_{\Gamma} G(x-y) \phi(y) \,dy = f(x)\quad\text{for almost all }x\in\Gamma.
\end{align}
Here, the fundamental solution $G$ stems from a strongly-elliptic PDE operator
\begin{align}\label{eq:PDEs}
\begin{split}
\mathfrak{P}u:=
-\sum_{i=1}^{d} \sum_{i'=1}^{d}\partial_i (A_{ii'}\partial_{i'} u)+\sum_{i=1}^{d} b_i\partial_i u +c u,
\end{split}
\end{align}
where the coefficients $A_{ii'}=\overline{A_{i'i}}^\top,b_i,c\in\C^{D\times D}$ are constant for some fixed dimension $D\ge 1$.


\subsection{Outline \& Contributions}
%
In Section~\ref{sec:preliminaries}, we fix some general notation, recall Sobolev spaces on the boundary, and precisely state the considered problem.
Section~\ref{sec:abstract setting bem} can be paraphrased as follows: 
We formulate an adaptive algorithm (Algorithm~\ref{alg:bem algorithm})  
 of the form 
\begin{align}
\boxed{\rm~SOLVE~} 
\longrightarrow \boxed{\rm~ESTIMATE~}
\longrightarrow \boxed{\rm~MARK~}
\longrightarrow \boxed{\rm~REFINE~}
\end{align}
driven by some weighted-residual {\sl a~posteriori} error estimator (see~\eqref{intro:rel} below) in the frame of conforming Galerkin BEM. 
The algorithm particularly generates meshes $\TT_\ell$, BEM solutions $\Phi_\ell$ in associated nested ansatz spaces $\XX_\ell\subseteq\XX_{\ell+1}\subset L^2(\Gamma)^D\subset H^{-1/2}(\Gamma)^D$, and error estimators $\eta_\ell$ for all $\ell\in\N_0$.
We formulate five assumptions~\eqref{M:patch bem}--\eqref{M:semi bem} on the underlying meshes (Section~\ref{subsec:boundary discrete bem}), five assumptions~\eqref{R:sons bem}--\eqref{R:overlay bem} on the mesh-refinement (Section~\ref{subsec:general refinement bem}), and six assumptions~\eqref{S:inverse bem}--\eqref{S:stab bem} on the 
BEM spaces (Section~\ref{subsec:ansatz}).
First, these assumptions are sufficient to guarantee that the {\sl a posteriori} error estimator $\eta_\ell$ associated with the BEM solution $\Phi_\ell$ is 
reliable, i.e., there exists a constant
$\const{rel}>0$ such that 
\begin{align}\label{intro:rel}
 \const{rel}^{-1} \norm{\phi-\Phi_\ell}{H^{-1/2}(\Gamma)} 
 \le \eta_\ell 
 :=\norm{h_\ell^{1/2} \nabla_\Gamma (f- \mathfrak{V}\Phi_\ell)}{L^2(\Gamma)}
 \quad \text{for all } \ell \in \N_0,
\end{align} 
where $h_\ell\in L^\infty(\Gamma)$ denotes the local mesh-size function and $\nabla_\Gamma$ is the surface gradient.
Second, Theorem~\ref{thm:abstract bem} states that Algorithm~\ref{alg:bem algorithm} leads to linear convergence at optimal algebraic rate with respect to the number of mesh elements.
In Theorem~\ref{thm:concrete main}, we briefly note that the introduced conditions have already been implicitly proved for standard discretizations with piecewise polynomials on conforming triangulations. 
Moreover, we mention expected applications to adaptive IGABEM on quadrilateral meshes in Remark~\ref{rem:concrete main}.

Section~\ref{sec:proof abstract bem} is devoted to the proof of Theorem~\ref{thm:abstract bem}.
To prove reliability~\eqref{intro:rel}, we use a localization argument (Proposition~\ref{prop:faermann}), which generalizes earlier works~\cite{faer2,faer3} for standard discretizations. 
More precisely, we prove that
\begin{align}
\norm{v}{H^{1/2}(\Gamma)}^2\le C_{\rm split}\sum_{T\in\TT_\ell}\sum_{T'\in \Pi_\ell(T)}|v|_{H^{1/2}(T\cup T')}^2 
\end{align}
for all $v\in H^{1/2}(\Gamma)^D$ that are $L^2$-orthogonal onto the ansatz space $\XX_\ell$ corresponding to some mesh $\TT_\ell$, where $\const{split}>0$ is independent of $v$ 
and $\Pi_\ell(T)$ denotes the patch of $T\in\TT_\ell$.
In Remark~\ref{rem:faermann}, we note that one obtains at least plain convergence $\lim_{\ell\to\infty}\norm{\phi-\Phi_\ell}{H^{-1/2}(\Gamma)}=0$ if Algorithm~\ref{alg:bem algorithm} is steered by the so-called \textit{Faermann estimator}
\begin{align*}
 \widetilde C_{\rm rel}^{-1} \norm{\phi-\Phi_\ell}{H^{-1/2}(\Gamma)}
 \le \digamma_{\hspace{-1mm} \ell}
 := \Big(\sum_{T\in\TT_\ell}\sum_{T'\in \Pi_\ell(T)} \!\! |f-\mathfrak{V}\Phi_\ell|_{H^{1/2}(T\cup T')}^2\Big)^{1/2}
 \!\!\le \widetilde C_{\rm eff}^{-1} \norm{\phi-\Phi_\ell}{H^{-1/2}(\Gamma)},\hspace*{-2mm}
\end{align*}
which is reliable and efficient. 
To prove linear convergence at optimal rate for the weighted-residual estimator~\eqref{intro:rel}, we show that the assumptions of Section~\ref{sec:abstract setting bem} imply the \textit{axioms of adaptivity}~\cite{axioms}. 
The latter are properties for  abstract mesh-refinements and abstract error estimators, which automatically yield the desired convergence result.
In contrast to \cite{fkmp,part1} which (implicitly) verify the axioms of adaptivity only for the Laplace problem, our analysis allows for general PDE operators~\eqref{eq:PDEs}.
The crucial step is the generalization (Proposition~\ref{prop:invest for V}) of the non-trivial \emph{local} inverse inequality for the \emph{non-local} boundary integral operator $\mathfrak{V}$: With the help of a Caccioppoli-type inequality (Lemma~\ref{lem:cacc}), we prove that there exists a constant $\const{inv} > 0$ such that
\begin{align}\label{eq:one invest}
\norm{h_\ell^{1/2}\nabla_\Gamma \mathfrak{V}\psi}{L^2(\Gamma)}\le \const{inv} \big(\norm{\psi}{H^{-1/2}(\Gamma)}+\norm{h_\ell^{1/2}\psi}{L^2(\Gamma)}\big)\quad\text{for all }\psi\in L^2(\Gamma)^D;
\end{align}
see~\cite{invest} for standard BEM discretizations of the Laplacian. Similar estimates hold also for the other boundary integral operators related to~\eqref{eq:PDEs}, namely the double-layer integral operator $\mathfrak{K}$, its adjoint $\mathfrak{K}'$, and the hypersingular integral operator $\mathfrak{W}$. These are stated and proved in Appendix~\ref{sec:proof general invest}; again we refer to~\cite{invest} for standard BEM discretizations of the Laplacian.

While the present work focusses on the numerical analysis aspects only, we refer to the literature (see, e.g., \cite{cms,gme,resigabem,diss,helmholtzbem}) for numerical experiments for the Laplace problem, the Helmholtz problem, and linear elasticity.

\section{Preliminaries}
\label{sec:preliminaries}
In this section, we fix some general notation, recall Sobolev spaces on the boundary, and precisely state the considered problem.
Throughout the work, 
let $\Omega\subset\R^d$ for $d\ge2$ be a bounded Lipschitz domain as in \cite[Definition~3.28]{mclean} and $\Gamma:= \partial\Omega$ its boundary.

\subsection{General notation}
\label{sec:general notation}
Throughout and without any ambiguity, $|\cdot|$ denotes the absolute value of scalars, the Euclidean norm of vectors in $\R^n$, as well as  the $d$-dimensional measure of a set in $\R^n$. 
Let $B_{\varepsilon}(x):=\set{y\in\R^n}{|x-y|<\varepsilon}$ denote the open ball around $x$ with radius $\varepsilon>0$. 
For $\emptyset\neq\omega_1,\omega_2\subseteq\R^n$, let $B_{\varepsilon}(\omega_1):=\bigcup_{x\in\omega_1} B_\varepsilon(x)$.
Moreover, let $\diam(\omega_1):=\sup\set{|x-y|}{x,y\in\omega_1}$, and 
$\dist(\omega_1,\omega_2):=\inf\set{|x-y|}{x\in\omega_1,y\in\omega_2}$.
We write $A\lesssim B$ to abbreviate $A \le CB$ with some generic constant $C > 0$, which is clear from the context. 
Moreover, $A \simeq  B$ abbreviates $A\lesssim B \lesssim A$. Throughout, mesh-related quantities have the same index, e.g., $\XX_\bullet$ is the ansatz space corresponding to the  mesh  $\TT_\bullet$. 
The analogous notation is used for meshes $\TT_\circ$, $\TT_\star$, $\TT_{\ell}$ etc.

\subsection{Sobolev spaces}\label{subsec:sobolev}
For $\sigma\in[0,1]$, we define the Hilbert spaces $H^{\pm\sigma}(\Gamma)$ as in \cite[page~99]{mclean} by use of Bessel potentials on $\R^{d-1}$ and liftings via bi-Lipschitz mappings\footnote{For $\widehat\omega\subseteq \R^{d-1}$ and $\omega\subseteq \R^d$, a mapping $\gamma:\widehat\omega\to\omega$ is bi-Lipschitz if it is bijective and $\gamma$ as well as its inverse $\gamma^{-1}$ are Lipschitz continuous.} that describe $\Gamma$.
For $\sigma=0$, it holds that $H^0(\Gamma)=L^2(\Gamma)$ with equivalent norms.
We thus may define $\norm{\cdot}{H^0(\Gamma)}:=\norm{\cdot}{L^2(\Gamma)}$.

For $\sigma\in(0,1]$,  any measurable subset $\omega\subseteq\Gamma$, and all $v\in H^\sigma(\Gamma)$, we define the associated   Sobolev--Slobodeckij norm
\begin{align}\label{eq:SS-norm}
 \norm{v}{H^{\sigma}(\omega)}^2
 := \norm{v}{L^2(\omega)}^2
 + |v|_{H^{\sigma}(\omega)}^2\text{ with }
 |v|_{H^{\sigma}(\omega)}^2 :=\begin{cases} \int_\omega\int_\omega\frac{|v(x)-v(y)|^2}{|x-y|^{d-1+2\sigma}}\,dx dy&\text{ if }\sigma\in(0,1),\\ \norm{\nabla_\Gamma v}{L^2(\omega)}^2&\text{ if }\sigma=1.\end{cases}
\end{align}
It is well-known that $\norm{\cdot}{H^\sigma(\Gamma)}$ provides an equivalent norm on $H^\sigma(\Gamma)$; see, e.g., \cite[Lemma~2.19]{s} and \cite[Theorem~3.30 and page 99]{mclean} for $\sigma\in(0,1)$ and
\cite[Theorem~2.28]{gme} for $\sigma=1$.
Here, $\nabla_\Gamma(\cdot)$ denotes the usual (weak) surface gradient which can be  defined for almost all $x\in\Gamma$ as follows:
Since $\Gamma$ is a Lipschitz boundary, there exist an open cover  $(O_j)_{j=1}^J$ in $\R^d$ of $\Gamma$ such that each $\omega_j:=O_j\cap \Gamma$ can be parametrized by  a bi-Lipschitz mapping
 $\gamma_{\omega_j}: \widehat\omega_j\to \omega_j$, where  $\widehat \omega_j\subset\R^{d-1}$  is an  open set.
By Rademacher's theorem, $\gamma_{\omega_j}$ is  almost everywhere differentiable. 
The corresponding Gram determinant $\det(D\gamma_{\omega_j}^\top D\gamma_{\omega_j})$ is almost everywhere positive. 
Moreover, by definition of the space $H^1(\Gamma)$, $v\in H^1(\Gamma)$ implies that $v\circ\gamma_{\omega_j}\in H^1(\widehat\omega_j)$.
With the weak derivative $\nabla (v\circ\gamma_{\omega_j})\in L^2(\widehat\omega_j)^d$, we can  hence define 
\begin{align}\label{eq:surface gradient}
(\nabla_\Gamma v )|_{\omega_j} :=\big(D\gamma_{\omega_j}(D\gamma_{\omega_j}^\top D\gamma_{\omega_j})^{-1} \nabla (v\circ\gamma_{\omega_j})\big)\circ\gamma_{\omega_j}^{-1}
\quad\text{for all }v\in H^1(\Gamma).
\end{align}
This definition does not depend on the particular choice of the open sets $(O_j)_{j=1}^J$
and the corresponding parametrizations $(\gamma_{\omega_j})_{j=1}^J$; see, e.g., \cite[Theorem~2.28]{gme}.
With \eqref{eq:surface gradient}, we immediately obtain the chain rule 
\begin{align}\label{eq:chain rule}
\nabla(v\circ\gamma_{\omega_j})=D\gamma_{\omega_j}^\top((\nabla_\Gamma v)\circ\gamma_{\omega_j}(\cdot))\quad\text{for all }v\in H^1(\Gamma).
\end{align}

For $\sigma\in(0,1]$, $H^{-\sigma}(\Gamma)$ is a realization of the dual space of $H^{\sigma}(\Gamma)$;  see \cite[Theorem~3.30 and page~99]{mclean}.
With the duality bracket $\dual{\cdot}{\cdot}$, we define an equivalent norm 
\begin{align}
\norm{\psi}{H^{-\sigma}(\Gamma)}:=\sup \set{\dual{v}{\psi}}{v\in H^\sigma(\Gamma)\wedge\norm{v}{H^\sigma(\Gamma)}=1} \quad\text{for all } \psi\in H^{-\sigma}(\Gamma).
\end{align}
Moreover, we abbreviate
\begin{align}
\sprod{v}{\psi}:=\dual{\overline v}{\psi}\quad \text{for all } v\in H^\sigma(\Gamma),\psi\in H^{-\sigma}(\Gamma).
\end{align}

\cite[page~76]{mclean} states that the inclusion $H^{\sigma_1}(\Gamma)\subseteq H^{\sigma_2}(\Gamma)$ for $-1\le \sigma_1\le\sigma_2\le 1$ is continuous and dense.
In particular, $H^{\sigma}(\Gamma)\subset L^2(\Gamma)\subset H^{-\sigma}(\Gamma)$ form a Gelfand triple in the sense of \cite[Section~2.1.2.4]{ss} for all $\sigma\in(0,1]$, where $\psi\in L^2(\Gamma)$ is interpreted as function in $H^{-\sigma}(\Gamma)$ via 
\begin{align}
\dual{v}{\psi}:=\sprod{\overline v}{\psi}_{L^2(\Gamma)}=\int_\Gamma v\,\psi \,dx \quad\text{for all }v\in H^\sigma(\Gamma),\psi\in L^2(\Gamma).
\end{align}
Here, $\sprod{\cdot}{\cdot}_{L^2(\Gamma)}$ is the usual complex scalar product on $L^2(\Gamma)$.

So far, we have only dealt with scalar-valued functions. 
For $D\ge1$, $\sigma\in[0,1]$, $v=(v_1,\dots,v_D)\in H^{\sigma}(\Gamma)^D$, we define  $\norm{v}{H^{\pm\sigma}(\Gamma)}^2:=\sum_{j=1}^D \norm{v_j}{H^{\pm\sigma}(\Gamma)}^2$. 
If $\sigma>0$, and $\omega\subseteq\Gamma$ is an arbitrary measurable set, we define $\norm{v}{H^\sigma(\omega)}$ and $\seminorm{v}{H^\sigma(\omega)}$ analogously.
With the definition
\begin{align}
\nabla_\Gamma v:=
\begin{pmatrix}
\nabla_\Gamma v_1\\
\vdots\\
\nabla_\Gamma v_D
\end{pmatrix}
\in L^2(\Gamma)^{D^2}\quad\text{for all }v\in H^1(\Gamma)^D,
\end{align}
it holds that
$\seminorm{v}{H^1(\omega)}=\norm{\nabla_\Gamma v}{L^2(\omega)}$.
Note that $H^{-\sigma}(\Gamma)^D$ with $\sigma\in(0,1]$ can be identified with the dual space of $H^{\sigma}(\Gamma)^D$, where we set
\begin{align}
\dual{v}{\psi}:=\sum_{j=1}^D\dual{v_j}{\psi_j}\quad\text{for all }v\in H^\sigma(\Gamma)^D,\psi\in H^{-\sigma}(\Gamma)^D.
\end{align}
As before, we abbreviate
\begin{align}
\sprod{v}{\psi}:=\dual{\overline v}{\psi}\quad \text{for all } v\in H^\sigma(\Gamma)^D,\psi\in H^{-\sigma}(\Gamma)^D
\end{align}
and set
\begin{align}
\dual{v}{\psi}:=\sprod{\overline v}{\psi}_{L^2(\Gamma)}=
\sum_{j=1}^D\int_\Gamma v_j \psi_j\,dx\quad\text{for all }v\in H^\sigma(\Gamma)^D,\psi\in L^2(\Gamma)^D.
\end{align}

{The spaces $H^\sigma(\Gamma)$ can also be defined as trace spaces or via interpolation, where the resulting norms are always equivalent with constants depending only on the dimension $d$ and the boundary $\Gamma$.}
More details and proofs are found, e.g., in the monographs  \cite{mclean,ss,s}.

\subsection{Continuous problem}\label{subsec:model problem bem}
We consider a general second-order linear  system of PDEs 
\begin{align}\label{eq:PDE bem}
\begin{split}
\mathfrak{P}u:=
-\sum_{i=1}^{d} \sum_{i'=1}^{d}\partial_i (A_{ii'}\partial_{i'} u)+\sum_{i=1}^{d} b_i\partial_i u +c u,
\end{split}
\end{align}
where the coefficients $A_{ii'},b_i,c\in\C^{D\times D}$ are constant for some fixed dimension $D\ge 1$.
We suppose that $A_{ii'}^\top=\overline{A_{i'i}}$.
Moreover, we  assume that $\mathfrak{P}$ is coercive on $H_0^1(\Omega)^D$, i.e., the sesquilinear form 
\begin{align}\label{eq:PDE form}
\sprod{u}{v}_{\mathfrak{P}}:=\int_\Omega \sum_{i=1}^d \sum_{i'=1}^d(A_{ii'}\partial_{i'} u)\cdot \partial_i v+\sum_{i=1}^d (b_i\partial_i u)\cdot v +(c u)\cdot v\,dx
\end{align}
is  elliptic up to some compact perturbation.
This is equivalent to \textit{strong ellipticity}
of the matrices $A_{ii'}$ in the sense of \cite[page~119]{mclean}.
Here, the standard complex scalar product on $\C^D$ is denoted by $w\cdot z=\sum_{j=1}^D \overline {w}_j z_j$.

Let $G:\R^d\setminus\{0\}\to \C^{D\times D}$ be a corresponding (matrix-valued) fundamental solution in the sense of \cite[page~198]{mclean}, i.e., a distributional solution of $\mathfrak{P}G=\delta$, where $\delta$ denotes the Dirac delta distribution.
For $\psi\in L^\infty(\Gamma)^D$, we define the \textit{single-layer operator}
as
\begin{align}\label{eq:single layer operator integral}
({\mathfrak{V}}\psi)(x):=\int_{\Gamma} G(x-y) \psi(y) \,dy\quad\text{for all }x\in\Gamma.
\end{align}
According to \cite[page 209 and 219--220]{mclean} and \cite[
Corollary~3.38]{mitrea}, 
 this operator can be extended for arbitrary $\sigma\in(-1/2,1/2$] to a bounded linear operator 
\begin{align}\label{eq:single layer operator}
\mathfrak{V}:H^{-1/2+\sigma}(\Gamma)^D\to H^{1/2+\sigma}(\Gamma)^D.
\end{align}
\cite[Theorem~7.6]{mclean} states that $\mathfrak{V}$ is always elliptic up to some compact perturbation.
We assume that it is elliptic even without perturbation, i.e., 
\begin{align}\label{eq:ellipticity bem}
\mathrm{Re}\,\sprod{\mathfrak{V}\psi}{\psi}
\ge \const{ell}\norm{\psi}{H^{-1/2}(\Gamma)}^2\quad\text{for all }\psi\in H^{-1/2}(\Gamma)^D.
\end{align}
This is particularly satisfied for the Laplace problem or for the Lam\'e problem, where the case $d=2$ requires an additional scaling of the geometry $\Omega$; see, e.g., \cite[Chapter~6]{s}.
Moreover, the sesquilinear form $\sprod{\mathfrak{V}\,\cdot}{\cdot}
$ is continuous due to \eqref{eq:single layer operator}, i.e.,  it holds with $\const{cont}:=\norm{\mathfrak{V}}{H^{-1/2}(\Gamma)^D\to H^{1/2}(\Gamma)^D}$ that 
\begin{align}\label{eq:continuity bem}
|\sprod{\mathfrak{V}\psi}{\xi}|
\le \const{cont}\norm{\psi}{H^{-1/2}(\Gamma)}\norm{\xi}{H^{-1/2}(\Gamma)} \quad\text{for all }\psi,\xi\in H^{-1/2}(\Gamma)^D.
\end{align}

Given a right-hand side $f\in H^{1}(\Gamma)^D$, 
we consider the boundary integral equation
\begin{align}\label{eq:strong}
 \mathfrak{V}\phi = f.
\end{align}
Such equations arise from (and are even equivalent to) the solution of Dirichlet problems of the form $\mathfrak{P} u=0$ in $\Omega$ with $u=g$ on $\Gamma$ for some $g\in H^{1/2}(\Gamma)^D$; see, e.g., \cite[page 226--229]{mclean} for more details.
 The Lax--Milgram lemma provides existence
and uniqueness of the solution $\phi\in H^{-1/2}(\Gamma_{})^D$ of the equivalent variational formulation of~\eqref{eq:strong}
\begin{align}\label{eq:weak}
\sprod{\mathfrak{V}\phi}{\psi}
=\sprod{f}{\psi}
 \quad\text{for all }\psi\in H^{-1/2}(\Gamma_{})^D.
\end{align}
In particular, we see that $\mathfrak{V}:H^{-1/2}(\Gamma)^D\to H^{1/2}(\Gamma)^D$ is an isomorphism.
In the Galerkin boundary element method, the test
space $H^{-1/2}(\Gamma_{})^D$ is replaced by some discrete subspace  $\XX_\bullet\subset {L^{2}(\Gamma_{})}^D\subset H^{-1/2}(\Gamma_{})^D$.
Again, the Lax--Milgram lemma guarantees existence and uniqueness of the solution
$\Phi_\bullet\in\XX_\bullet$ of the discrete variational formulation
\begin{align}\label{eq:discrete}
\sprod{\mathfrak{V}\Phi_\bullet}{\Psi_\bullet}
 = \sprod{f}{\Psi_\bullet}
 \quad\text{for all }\Psi_\bullet\in\XX_\bullet.
\end{align}
Moreover, $\Phi_\bullet$ can in fact be computed by solving a linear system of equations.
Note that \eqref{eq:single layer operator} implies that $\mathfrak{V} \Psi_\coarse\in H^1(\Gamma)^D$ for arbitrary $\Psi_\coarse\in \XX_\coarse$.
The additional regularity $f\in H^1(\Gamma)^D$ instead of $f\in H^{1/2}(\Gamma)^D$ is only needed to define the residual error estimator~\eqref{eq:eta bem} below.
For a more detailed introduction to boundary integral equations, the reader is referred to the monographs \cite{mclean,ss,s}.


\section{Axioms of adaptivity (revisited)}\label{sec:axioms revisited bem}

The aim of this section is to formulate an adaptive algorithm (Algorithm~\ref{alg:bem algorithm}) for conforming BEM discretizations of our model problem~\eqref{eq:strong}, where adaptivity is driven by the \textit{residual {\sl a~posteriori} error estimator} (see \eqref{eq:eta bem} below).
We identify the crucial properties of the underlying meshes, the mesh-refinement, as well as the boundary element spaces which ensure that the residual error estimator fits into the general framework of \cite{axioms} and which hence guarantee optimal convergence behavior of the adaptive algorithm.
We mention that we have already identified similar (but not identical) properties for the finite element method in \cite[Section~3]{igafem}.
The main result of this work is Theorem~\ref{thm:abstract bem} which is proved in Section~\ref{sec:proof abstract bem}.
\label{sec:abstract setting bem}

\subsection{Meshes}\label{subsec:boundary discrete bem}
Throughout, $\TT_\coarse$ is a \textit{mesh} of the boundary $\Gamma=\partial\Omega$ of the bounded Lipschitz domain $\Omega\subset\R^d$ in the following sense:
\begin{enumerate}[\rm (i)]
\item $\TT_\coarse$ is a finite set of compact Lipschitz domains on $\Gamma$, i.e., each element $T$ has the form $T=\gamma_T({\widehat{T}})$, where ${\widehat T}$ is a compact\footnote{A compact Lipschitz domain is the closure of a bounded Lipschitz domain.
For $d=2$, it is the finite union of compact intervals with non-empty interior.} Lipschitz domain in $\R^{d-1}$ and $\gamma_T:{ \widehat T}\to T$ is bi-Lipschitz.
\item  $\TT_\coarse$ covers $\Gamma$, i.e., $\Gamma = \bigcup_{T\in\TT_\coarse}{T}$.
\item For all $T,T'\in\TT_\coarse$ with $T\neq T'$, the intersection $T\cap T'$ has $(d-1)$-dimensional Hausdorff measure  zero.
\end{enumerate}
We suppose that there is a countably infinite set $\T$ of \textit{admissible} meshes. 
In order to ease notation, we introduce for $\TT_\coarse\in\T$ the corresponding \textit{mesh-width function} 
\begin{align}
h_\coarse\in L^\infty(\Gamma)\quad\text{with}\quad h_\coarse|_T=h_T:=|T|^{1/(d-1)}\text{ for all }T\in\TT_\coarse.
\end{align}
 For  $\omega\subseteq \Gamma$, we define the patches of order $q\in\N_0$ inductively by
\begin{align}
 \pi_\bullet^0(\omega) := \omega,
 \quad 
 \pi_\bullet^q(\omega) := \bigcup\set{T\in\TT_\bullet}{ {T}\cap \pi_\bullet^{q-1}(\omega)\neq \emptyset}.
\end{align}
The corresponding set of elements is
\begin{align}
 \Pi_\bullet^q(\omega) := \set{T\in\TT_\bullet}{ {T} \subseteq \pi_\bullet^q(\omega)},
 \quad\text{i.e.,}\quad
 \pi_\bullet^q(\omega) = \bigcup\Pi_\bullet^q(\omega).
\end{align}
If $\omega=\{z\}$ for some $z\in\Gamma$, we simply write  $\pi^q_\bullet(z):=\pi^q_\bullet(\{z\})$ and $\Pi_\bullet^q(z) := \Pi_\bullet^q(\{z\})$.
  For $\SS\subseteq\TT_\bullet$, we define $\pi_\bullet^q(\SS):=\pi_\bullet^q(\bigcup\SS)$
and $\Pi_\bullet^q(\SS):=\Pi_\bullet^q(\bigcup\SS)$. 
To abbreviate notation, the index $q=1$ is omitted, e.g., $\pi_\bullet(\omega) := \pi_\bullet^1(\omega)$ and $\Pi_\bullet(\omega) := \Pi_\bullet^1(\omega)$.

We  assume the existence of constants $\const{patch},\const{locuni},\const{shape},\const{cent},\const{semi}>0$ such that the following assumptions are satisfied for all $\TT_\bullet\in\T$:
\begin{enumerate}[(i)]
\renewcommand{\theenumi}{M\arabic{enumi}}
\bf\item\rm\label{M:patch bem}
{\bf Bounded   element patch:} 
The number of elements in a patch is uniformly bounded, i.e., 
\begin{align*}
\#\Pi_\bullet(T)\le \const{patch}\quad \text{for all } T\in\TT_\bullet.
\end{align*}
\bf\item\rm\label{M:locuni bem}
{\bf Local quasi-uniformity:}
Neighboring elements have comparable diameter, i.e.,
\begin{align*}
{\diam(T)}/{\diam(T')}\le \const{locuni} \quad\text{for all }T\in\TT_\coarse\text{ and all }T'\in\Pi_\coarse(T).
\end{align*}
\bf\item\rm\label{M:shape bem}
{\bf Shape-regularity:} 
It holds that 
\begin{align*}
\const{shape}^{-1}\le{\diam(T)}/{h_T}\le \const{shape} \quad \text{for all }T\in\TT_\coarse.
\end{align*}
\bf\item\rm\label{M:cent bem}
{\bf Elements lie in the center of their patches:} 
It holds\footnote{We use the convention $\dist(T,\emptyset):=\diam(\Gamma)$.} that
\begin{align*}
\diam(T)\le C_{\rm cent} \, \dist(T,\Gamma\setminus\pi_\bullet(T))
\quad\text{for all }T\in\TT_\bullet.
\end{align*}
\bf\item\rm\label{M:semi bem}
{\bf Local seminorm estimate:}
For all $v\in H^1(\Gamma)$, it holds that
 \begin{align*}
|v|_{H^{1/2}(\pi_{\bullet}(z))}\le \const{semi} \,\diam(\pi_{\bullet}(z))^{1/2}| v|_{H^1(\pi_{\bullet}(z))}
\quad\text{for all }z\in\Gamma.
\end{align*}
\end{enumerate}

The following proposition shows that \eqref{M:semi bem} is actually always satisfied.
However, in general the  multiplicative constant  depends  on the shape of the point patches.
The proof is inspired by \cite[Proposition 2.2]{hitchhiker}, where an analogous assertion for norms instead of seminorms is found.
For $\sigma=1/2$ and $d=2$,  we have already shown  the assertion in the recent own work \cite[Lemma~4.5]{resigabem}.
For polyhedral domains $\Omega$ with triangular meshes, it is  proved in \cite[Proposition~3.3]{faerconv} via interpolation techniques. 
A detailed proof for our setting is found in \cite[Proposition~5.2.2]{diss}, 
where we essentially follow the proof of \cite[Lemma~4.5]{resigabem}.

\begin{proposition}\label{prop:semi estimate}
Let $\widehat \omega\subset\R^{d-1}$ be a bounded and  connected Lipschitz domain and  $\gamma_\omega:\widehat \omega \to\omega\subseteq\Gamma$ be bi-Lipschitz, i.e., there exists a constant $\const{lipref}>0$ 
such that
\begin{align}\label{eq:Clipref}
C_{\rm lipref}^{-1} |s-t|\le \frac{|\gamma_\omega(s)-\gamma_\omega(t)|}{\diam(\omega)}\le C_{\rm lipref} |s-t|\quad\text{for all }s,t\in\widehat\omega.
\end{align}
Then, for arbitrary $\sigma\in(0,1)$ there exists a  constant $\const{semi}(\widehat\omega)>0$ such that 
\begin{align}\label{eq:local estimate}
|v|_{H^\sigma(\omega)}\le \const{semi}(\widehat\omega) \,\diam(\omega)^{1-\sigma}| v|_{H^1(\omega)}\quad\text{for all }v\in H^1(\Gamma).
\end{align}
The constant $\const{semi}(\widehat\omega)>0$  depends only on the dimension $d$, $\sigma$, the set $\widehat\omega$, and $\const{lipref}$. \hfill$\square$
\end{proposition}

\subsection{Mesh-refinement}
\label{subsec:general refinement bem}
For $\TT_\coarse\in\T$ and an arbitrary set of marked elements $\MM_\coarse\subseteq\TT_\coarse$, we associate a corresponding \textit{refinement} $\TT_\fine:=\refine(\TT_\coarse,\MM_\coarse) \in\T$ with $\MM_\coarse\subseteq\TT_\coarse\setminus\TT_\fine$, i.e., at least the marked elements are  refined.
Moreover, we suppose for the cardinalities that $\#\TT_\coarse<\#\TT_\fine$ if $\MM_\coarse\neq\emptyset$ and  $\TT_\fine=\TT_\coarse$ else.
Let $\refine(\TT_\coarse)\subseteq\T$ be the set of all $\TT_\fine$ such that there exist  meshes $\TT_{(0)},\dots,\TT_{(J)}$ and marked elements $\MM_{(0)},\dots,\MM_{(J-1)}$ with $\TT_\fine=\TT_{(J)}=\refine(\TT_{(J-1)},\MM_{(J-1)}),\dots,\TT_{(1)}=\refine(\TT_{(0)},\MM_{(0)})$ and $\TT_{(0)}=\TT_\coarse$. 
We assume that there exists a fixed initial mesh $\TT_0\in\T$ with $\T=\refine(\TT_0)$.

We suppose that there exist $\const{son}\ge2$ and $0<\ro{son}<1$ such that all meshes $\TT_\coarse\in\T$  satisfy for arbitrary marked elements $\MM_\coarse\subseteq\TT_\coarse$ with corresponding refinement $\TT_\fine:=\refine(\TT_\coarse,\MM_\coarse)$, the following elementary properties~\eqref{R:sons bem}--\eqref{R:reduction bem}:
\begin{enumerate}[(i)]
\renewcommand{\theenumi}{R\arabic{enumi}}
\bf\item\rm\label{R:sons bem}
\textbf{Son estimate:}
One step of refinement leads to a bounded increase of elements, i.e., 
\begin{align*}\#\TT_\fine \le \const{son}\,\#\TT_\coarse,
\end{align*}
%
%
\bf\item\rm\label{R:union bem}
\textbf{Father is union of sons:}
Each element is the union of its successors, i.e., 
\begin{align*}
T=\bigcup\set{{T'}\in\TT_\fine}{T'\subseteq T} \quad\text{for all }T\in\TT_\coarse.
\end{align*}
\bf\item\rm\label{R:reduction bem}
\textbf{Reduction of sons:}
Successors are uniformly smaller than their father, i.e., 
\begin{align*}
|T'| \le \ro{son}\,|T|\quad\text{for all }T\in\TT_\coarse\text{ and all }T'\in\TT_\fine\text{ with } T'\subsetneqq T.
\end{align*}
\end{enumerate}
By induction and the definition of $\refine(\TT_\coarse)$, one easily sees that \eqref{R:union bem}--\eqref{R:reduction bem} remain valid if $\TT_\fine$ is an arbitrary mesh in $\refine(\TT_\coarse)$.
In particular, \eqref{R:union bem}--\eqref{R:reduction bem} imply that each refined element $T\in\TT_\coarse\setminus\TT_\fine$ is split into at least two sons, wherefore 
\begin{align}\label{eq:R:refine bem}
\#(\TT_\coarse\setminus\TT_\fine)\le \#\TT_\fine-\#\TT_\coarse\quad\text{for all }\TT_\fine\in\refine(\TT_\coarse).
\end{align}
Besides~\eqref{R:sons bem}--\eqref{R:reduction bem}, we suppose the following less trivial requirements~\eqref{R:closure bem}--\eqref{R:overlay bem} with  generic constants $\const{clos},\const{over}>0$:
\begin{enumerate}[(i)]
\renewcommand{\theenumi}{R\arabic{enumi}}
\setcounter{enumi}{3}
\bf\item\rm\label{R:closure bem}
\textbf{Closure estimate:}
Let $(\TT_\ell)_{\ell\in\N_0}$ be a sequence in $\T$ such that  $\TT_{\ell+1}=\refine(\TT_\ell,\MM_\ell)$ with some $\MM_\ell\subseteq \TT_\ell$ for all $\ell\in\N_0$.
Then, it holds  that 
\begin{align*}
\# \TT_\ell-\#\TT_0\le \const{clos}\sum_{j=0}^{\ell-1}\#\MM_j\quad\text{for all }\ell\in\N_0.
\end{align*}
\bf\item\rm\label{R:overlay bem}
\textbf{Overlay property:}
For all $\TT_\coarse,\TT_\star\in\T$, there exists a common refinement $\TT_\fine\in\refine(\TT_\coarse)\cap\refine(\TT_\star)$ which satisfies the overlay estimate
\begin{align*}
\#\TT_\fine \le \const{over}(\#\TT_\star - \#\TT_0)+\#\TT_\coarse.
\end{align*}
\end{enumerate}

\subsection{Boundary element space}\label{subsec:ansatz}
With each $\TT_\coarse\in\T$, we associate a finite dimensional space of vector valued functions
\begin{align}
\XX_\coarse \subset L^2(\Gamma)^D\subset H^{-1/2}(\Gamma)^D.
\end{align}
Let $\Phi_\coarse\in\XX_\coarse$ be the corresponding Galerkin approximation of $\phi\in H^{-1/2}(\Gamma)^D$ from \eqref{eq:strong}, i.e.,
\begin{align}\label{eq:pregalerkin bem}
 \sprod{\mathfrak{V}\Phi_\coarse}{\Psi_\coarse} = \sprod{f}{\Psi_\coarse}
 \quad\text{for all }\Psi_\coarse\in\XX_\coarse.
\end{align}
We note the Galerkin orthogonality
\begin{align}\label{eq:galerkin bem}
 \sprod{f-\mathfrak{V}\Phi_\coarse}{\Psi_\coarse} = 0
 \quad\text{for all }\Psi_\coarse\in\XX_\coarse,
\end{align}
as well as the resulting C\'ea type quasi-optimality
\begin{align}\label{eq:cea bem}
 \norm{\phi-\Phi_\coarse}{H^{-1/2}(\Gamma)}
 \le C_{\text{C\'ea}}\min_{\Psi_\coarse\in\XX_\coarse}\norm{\phi-\Psi_\coarse}{H^{-1/2}(\Gamma)}\quad
 \text{with}\quad
 C_{\text{C\'ea}} := \const{cont}/\const{ell}.
\end{align}%

We assume the existence of constants $\const{inv}>0$,  $\q{loc},\q{proj},\q{supp} \in\N_0$, and $0<\ro{unity}<1$  such that the following properties~\eqref{S:inverse bem}--\eqref{S:unity bem} hold for all $\TT_\coarse\in\T$:
\begin{enumerate}[(i)]
\renewcommand{\theenumi}{S\arabic{enumi}}
\bf\item\rm\label{S:inverse bem}
\textbf{Inverse inequality:}
For  all  $\Psi_\coarse\in\XX_\coarse$, it holds that 
\begin{align*}
\norm{h_\bullet^{1/2}\Psi_\coarse}{L^2(\Gamma)}\le \const{inv} \, \norm{\Psi_\coarse}{H^{-1/2}(\Gamma)}.
\end{align*}
\bf\item\rm\label{S:nestedness bem}
\textbf{Nestedness:}
For  all $\TT_\fine\in\refine(\TT_\coarse)$, it holds that 
\begin{align*}
\XX_\coarse\subseteq\XX_\fine.
\end{align*}
\bf\item\rm\label{S:local bem}
\textbf{Local domain of definition:}
For all $\TT_\fine\in\refine(\TT_\coarse)$, 
 $T\in\TT_\coarse\setminus \Pi_\coarse^{\q{loc}}( \TT_\coarse\setminus\TT_\fine)\subseteq\TT_\coarse\cap\TT_\fine$, and $\Psi_\fine\in\XX_\fine$, it holds that 
\begin{align*}\Psi_\fine|_{\pi_\coarse^{\q{proj}}(T)} \in \set{\Psi_\coarse|_{\pi_\coarse^{\q{proj}}(T)}}{\Psi_\coarse\in\XX_\coarse}.
\end{align*}
\bf\item\rm\label{S:unity bem}{\bf Componentwise local approximation of unity:} For all $T\in\TT_\bullet$ and all $j\in\{1,\dots,\D\}$, 
there exists some $\Psi_{\bullet,T,j}\in\XX_\bullet$ with  
\begin{align*}
T\subseteq \supp (\Psi_{\bullet,T,j})\subseteq \pi_\bullet^{\q{supp}}(T),
\end{align*}
 such that only the $j$-th component  does not vanish, i.e., 
 \begin{align*}
 (\Psi_{\bullet,T,j})_{j'}=0\quad\text{for}\quad j'\neq j,
 \end{align*}
 and
\begin{align*}
\norm{1-(\Psi_{\bullet,T,j})_j}{L^2(\supp(\Psi_{\bullet,T,j}))} \le \ro{unity}{|\supp(\Psi_{\bullet,T,j})_j|}^{1/2}.
\end{align*}
\end{enumerate}
\begin{remark}\label{rem:approx one}
Clearly, \eqref{S:unity bem} is in particular satisfied if $\XX_\bullet$ is a product space, i.e., $\XX_\bullet=\prod_{j=1}^{\D}(\XX_{\bullet})_j$, and each component $(\XX_{\bullet})_j\subset L^2(\Gamma)$ satisfies \eqref{S:unity bem}.
\end{remark}

Besides \eqref{S:inverse bem}--\eqref{S:unity bem}, we suppose that there exist constants $\const{sz}>0$ as well as $\q{sz}\in\N_0$ such that for all $\TT_\coarse\in\T$ and  $\SS\subseteq \TT_\bullet$, there exists a linear Scott--Zhang-type operator  $J_{\coarse,\SS}:L^2(\Gamma)^{\D}\to\set{\Psi_\coarse\in\XX_\coarse}{\Psi_\coarse|_{\bigcup(\TT_\coarse\setminus\SS)}=0}$ with the following properties~\eqref{S:proj bem}--\eqref{S:stab bem}:
\begin{enumerate}[(i)]
\renewcommand{\theenumi}{S\arabic{enumi}}
\setcounter{enumi}{4}
\bf\item\rm\label{S:proj bem}
\textbf{Local projection property.}
Let $\q{loc}, \q{proj}\in\N_0$ from \eqref{S:local bem}.
For all $\psi\in L^2(\Gamma)^{\D}$ and $T\in\TT_\coarse$ with $\Pi_\coarse^{\q{loc}}(T)\subseteq\SS$, it holds  that 
\begin{align*}
(J_{\coarse,\SS} \psi)|_T = \psi|_T \quad\text{if }\psi|_{\pi_\coarse^{\q{proj}}(T)} \in \set{\Psi_\coarse|_{\pi_\coarse^{{\q{proj}}}(T)}}{\Psi_\coarse\in\XX_\coarse}.
\end{align*}

\bf\item\rm\label{S:stab bem}
\textbf{Local $\boldsymbol{L^2}$-stability.}
For all   $\psi\in L^2(\Gamma)^{\D}$ and $T\in\TT_\coarse$, it holds that 
\begin{align*}
\norm{ J_{\coarse,\SS} \psi}{L^2(T)}\le \const{sz} \norm{\psi}{L^2(\pi_\coarse^{\q{sz}}(T))}.
\end{align*}
\end{enumerate}

\subsection{Error estimator}\label{subsec:estimator bem}
Let $\TT_\coarse\in\T$.
Due to the regularity assumption $f\in H^1(\Gamma)^D$, the mapping property \eqref{eq:single layer operator}, and $\XX_\coarse\subset L^2(\Gamma)^D$, it holds that $f-\mathfrak{V}\Psi_\coarse\in H^1(\Gamma)^D$ for all $\Psi_\coarse\in\XX_\coarse$.
This allows to employ the  weighted-residual {\sl a~posteriori} error estimator
\begin{subequations}\label{eq:eta bem}
\begin{align}
 \eta_\coarse := \eta_\coarse(\TT_\coarse)
 \quad\text{with}\quad 
 \eta_\coarse(\SS)^2:=\sum_{T\in\SS} \eta_\coarse(T)^2
 \text{ for all }\SS\subseteq\TT_\coarse,
\end{align}
where the local refinement indicators read
\begin{align}
\eta_\coarse(T)^2:=h_T \seminorm{f-\mathfrak{V}\Phi_\coarse}{H^1(T)}^2\quad\text{for all }T\in\TT_\coarse.
\end{align}
\end{subequations}
The latter estimator goes back to the works \cite{cs96,c97}, where reliability \eqref{eq:reliable bem} is proved for standard 2D BEM with piecewise polynomials on polygonal geometries, while the corresponding result for 3D BEM is found in \cite{cms}.

\subsection{Adaptive algorithm}
We consider the following concrete realization of the abstract algorithm from  \cite[Algorithm~2.2]{axioms}.
\begin{algorithm}
\label{alg:bem algorithm}
\textbf{Input:} 
\hspace{-1.2mm}D\"orfler parameter $\theta\in(0,1]$ and marking constant $\const{min}\in[1,\infty]$.\\
\textbf{Loop:} For each $\ell=0,1,2,\dots$, iterate the following steps:
\begin{itemize}
\item[\rm(i)] Compute Galerkin approximation $\Phi_\ell\in\XX_\ell$.
\item[\rm(ii)] Compute refinement indicators $\eta_\ell({T})$
for all elements ${T}\in\TT_\ell$.
\item[\rm(iii)] Determine a set of marked elements $\MM_\ell\subseteq\TT_\ell$ which has up to the multiplicative constant $\const{min}$  minimal cardinality, such that the following D\"orfler marking is satisfied
\begin{align}
 \theta\,\eta_\ell^2 \le \eta_\ell(\MM_\ell)^2.
 \end{align}
\item[\rm(iv)] Generate refined mesh $\TT_{\ell+1}:=\refine(\TT_\ell,\MM_\ell)$. 
\end{itemize}
\textbf{Output:} Refined meshes $\TT_\ell$ and corresponding Galerkin approximations $\Phi_\ell$ with error estimators $\eta_\ell$ for all $\ell \in \N_0$.
\end{algorithm}

\subsection{Optimal convergence}
\label{sec:axioms}
Define
\begin{align}
 \T(N):=\set{\TT_\coarse\in\T}{\#\TT_\coarse-\#\TT_0\le N}
 \quad\text{for all }N\in\N_0
\end{align}
and for all $s>0$
\begin{align}
\const{approx}(s)
 := \sup_{N\in\N_0}\min_{\TT_\coarse\in\T(N)}(N+1)^s\,\eta_\coarse\in[0,\infty].
\end{align}
We say that the solution $\phi\in H^{-1/2}(\Gamma)^D$ lies in the \textit{approximation class $s$ with respect to the estimator} if
\begin{align}
\norm{\phi}{\mathbb{A}_s^{\rm est}} :=\const{approx}(s)<\infty.
\end{align}
By definition, $\norm{\phi}{\mathbb{A}^{\rm est}_s}<\infty$  implies that the error estimator $\eta_\coarse$
on the optimal meshes $\TT_\coarse$ decays at least with rate $\OO\big((\#\TT_\coarse)^{-s}\big)$. The following main theorem states that each possible rate $s>0$ is in fact realized by Algorithm~\ref{alg:bem algorithm}.
The proof is given in Section~\ref{sec:proof abstract bem}.
It essentially follows   by verifying the \textit{axioms of adaptivity} from \cite{axioms}.
Such an optimality result was first proved in \cite{fkmp} for the Laplace operator $\mathfrak{P}=-\Delta$ on a polyhedral domain $\Omega$.
As ansatz space, they considered piecewise constants on shape-regular triangulations.
\cite{part1} in combination with \cite{invest} extends the assertion to piecewise polynomials on shape-regular curvilinear triangulations  of some piecewise smooth boundary $\Gamma$.
Independently, \cite{gantumur} proved the same result for globally smooth $\Gamma$ and general self-adjoint and elliptic boundary integral operators.

\begin{theorem}\label{thm:abstract bem}
Let $(\TT_\ell)_{\ell\in\N_0}$ be the sequence of meshes generated by Algorithm~\ref{alg:bem algorithm}.
Then, there hold the following assertions {\rm (i)--(iii)}:
\begin{enumerate}[\rm (i)]
\item\label{item:qabstract reliable bem}
Suppose \eqref{M:patch bem}--\eqref{M:semi bem} and \eqref{S:unity bem}.
Then, 
the residual error estimator satisfies reliability, i.e., there exists a constant $\const{rel}>0$ such that
\begin{align}\label{eq:reliable bem}
 \norm{\phi-\Phi_\coarse}{H^{-1/2}(\Gamma)}\le \const{rel}\eta_\coarse\quad\text{for all }\TT_\coarse\in\T.
\end{align}
\item\label{item:qabstract linear convergence bem}
Suppose  \eqref{M:patch bem}--\eqref{M:semi bem}, \eqref{R:union bem}--\eqref{R:reduction bem},  \eqref{S:inverse bem}--\eqref{S:nestedness bem}, and \eqref{S:unity bem}.
Then, for arbitrary $0<\theta\le1$ and $\const{min}\in[1,\infty]$, the   estimator converges linearly, i.e., there exist constants $0<\ro{lin}<1$ and $\const{lin}\ge1$ such that
\begin{align}\label{eq:linear bem}
\eta_{\ell+j}^2\le \const{lin}\ro{lin}^j\eta_\ell^2\quad\text{for all }j,\ell\in\N_0.
\end{align}
\item\label{item:qabstract optimal convergence bem}
Suppose \eqref{M:patch bem}--\eqref{M:semi bem},  \eqref{R:sons bem}--\eqref{R:overlay bem}, and \eqref{S:inverse bem}--\eqref{S:stab bem}.
Then, there exists a constant $0<\theta_{\rm opt}\le1$ such that for all $0<\theta<\theta_{\rm opt}$ and $\const{min}\in[1,\infty)$, the estimator converges at optimal rate, i.e., for all $s>0$ there exist constants $c_{\rm opt},\const{opt}>0$ such that
\begin{align}\label{eq:optimal bem}
 c_{\rm opt}\norm{\phi}{\mathbb{A}_s^{\rm est}}
 \le \sup_{\ell\in\N_0}{(\# \TT_\ell-\#\TT_0+1)^{s}}\,{\eta_\ell}
 \le \const{opt}\norm{\phi}{\mathbb{A}_s^{\rm est}},
\end{align}
where the lower bound requires only \eqref{R:sons bem} to hold.
\end{enumerate}
\noindent All involved constants $\const{rel},\const{lin},q_{\rm lin},\theta_{\rm opt}$, and $\const{opt}$ depend only on the assumptions made as well as the dimensions $d,D$, the coefficients of the differential operator $\mathfrak{P}$, and $\Gamma$, while $\const{lin},\ro{lin}$ depend additionally on $\theta$ and the sequence $(\Phi_\ell)_{\ell\in\N_0}$, and $\const{opt}$ depends furthermore on $\const{min}$, and $s>0$.
The constant $c_{\rm opt}$ depends only on $\const{son}, \#\TT_0$,  $s$, and if there exists $\ell_0$ with $\eta_{\ell_0}=0$, then also on $\ell_0$ and $\eta_0$.
\end{theorem}

\begin{remark}
If the sesquilinear form $\sprod{\mathfrak{V}\,\cdot}{\cdot}$ is Hermitian, then $\const{lin}$, $\ro{lin},$ and  $\const{opt}$ are  independent of $(\Phi_\ell)_{\ell\in\N_0}$; see Remark~\ref{rem:E3 bem} below.
\end{remark}

\begin{remark}
Let $\Gamma_0\subsetneqq\Gamma$ be an open subset of $\Gamma=\partial\Omega$ and let $\mathfrak{E}_0:L^2(\Gamma_0)^D\to L^2(\Gamma)^D$ denote the operator that extends a function defined on $\Gamma_0$ to a function on $\Gamma$ by zero.
We define the space of restrictions $H^{1/2}(\Gamma_0):=\set{v|_{\Gamma_0}}{v\in H^{1/2}(\Gamma)}$ endowed with the quotient norm $v_0\mapsto\inf\set{\norm{v}{H^{1/2}(\Gamma)}}{v|_{\Gamma_0}=v_0}$ 
 and its dual space $\widetilde H^{-1/2}(\Gamma_0):= H^{1/2}(\Gamma_0)^*$.
According to \cite[Section~2.1]{invest}, $\mathfrak{E}_0$ can be extended to an isometric operator $\mathfrak{E}_0:\widetilde H^{-1/2}(\Gamma_0)^D\to  H^{-1/2}(\Gamma)^D$.
Then, one can consider the integral equation
\begin{align}\label{eq:screen}
(\mathfrak{V}\mathfrak{E}_0\phi)|_{\Gamma_0}=f|_{\Gamma_0},
\end{align}
where $(\mathfrak{V}\mathfrak{E}_0(\cdot))|_{\Gamma_0}:\widetilde H^{-1/2}(\Gamma_0)^D\to H^{1/2}(\Gamma_0)^D$.
In the literature, such problems are known as \emph{screen problems}; see, e.g., \cite[Section~3.5.3]{ss}.
 Theorem~\ref{thm:abstract bem}  holds analogously for the screen problem \eqref{eq:screen}.
Indeed, the works \cite{fkmp,part1,invest,gantumur} cover this case as well.
However, to ease the presentation, we  focus  on closed boundaries $\Gamma_0=\Gamma=\partial\Omega$.
\end{remark}

\begin{remark}
{\rm(a)}
Let us additionally assume that $\XX_\coarse$ contains all componentwise constant functions, i.e.,  
\begin{align}\label{eq:alexass}
x\in\XX_\coarse
\quad\text{for all }x\in\C^D.
\end{align}
Then, under the assumption that $\norm{h_{\ell}}{L^\infty(\Omega)}\to0$ as $\ell\to\infty$, one can show  that
$\XX_\infty:=\overline{\bigcup_{\ell\in\N_0}\XX_\ell}=H^{-1/2}(\Gamma)^D$.
To see this, recall that $H^{1/2}(\Gamma)^D$ is  continuously and densely embedded in $L^2(\Gamma)^D$ which is itself continuously and densely embedded in $H^{-1/2}(\Gamma)^D$.
For $\psi\in H^{-1/2}(\Gamma)^D$ and arbitrary $\varepsilon>0$, let $\psi_{\varepsilon}\in H^{1/2}(\Gamma)^D$ with $\norm{\psi-\psi_{\varepsilon}}{H^{-1/2}(\Gamma)}\le\varepsilon$.
We abbreviate the  projection operator $J_\ell:=J_{\ell,\TT_\ell}$ for all $\ell\in\N_0$.
For all $T\in\TT_\ell$, the projection property~\eqref{S:proj bem} in combination with our additional assumption \eqref{eq:alexass}, the triangle inequality, and the local $L^2$-stability \eqref{S:stab bem} show that
\begin{align*}
\norm{(1-J_\ell)\psi_{\varepsilon}}{L^2(T)}&\stackrel{\eqref{S:proj bem}}=\Big\|(1-J_\ell)\Big(\psi_{\varepsilon} -\frac{1}{|\pi_\coarse^{\q{sz}}(T)|}\int_{\pi_\coarse^{\q{sz}}(T)} \psi_{\varepsilon}\,dx\Big)\Big\|_{L^2(T)}\\
&\stackrel{\eqref{S:stab bem}}\le (1+\const{sz}) \Big\|\psi_{\varepsilon} -\frac{1}{|\pi_\coarse^{\q{sz}}(T)|}\int_{\pi_\coarse^{\q{sz}}(T)} \psi_{\varepsilon}\,dx\Big\|_{L^2(\pi_\coarse^{\q{sz}}(T))}.
\end{align*}
With this, the Poincar\'e-type inequality from Lemma~\ref{lem:Poincare} below, and  \eqref{M:patch bem}--\eqref{M:shape bem}, we see that
\begin{align*}
\norm{(1-J_\ell)\psi_{\varepsilon}}{L^2(T)}\lesssim h_T^{1/2}|\psi_{\varepsilon}|_{H^{1/2}(\pi_\coarse^{\q{sz}}(T))}\le \norm{h_\ell}{L^\infty(\Gamma)}^{1/2}|\psi_{\varepsilon}|_{H^{1/2}(\pi_\coarse^{\q{sz}}(T))}.
\end{align*}
Summing over all elements, we obtain that
\begin{align*}
\norm{(1-J_\ell)\psi_{\varepsilon}}{H^{-1/2}(\Gamma)}^2\lesssim\norm{(1-J_\ell)\psi_{\varepsilon}}{L^{2}(\Gamma)}^2\lesssim \norm{h_\ell}{L^\infty(\Gamma)} \sum_{T\in \TT_\coarse} \seminorm{\psi_{\varepsilon}}{H^{1/2}(\pi_\coarse^{\q{sz}}(T))}^2.
\end{align*}
With \eqref{M:patch bem}--\eqref{M:cent bem}, Proposition~\ref{prop:easy faermann} and Lemma~\ref{lem:patch to elements} from below prove that   $\sum_{T\in \TT_\coarse} \seminorm{\psi_{\varepsilon}}{H^{1/2}(\pi_\coarse^{\q{sz}}(T))}^2$ $\lesssim$ $\seminorm{\psi_{\varepsilon}}{H^{1/2}(\Gamma)}^2$. 
Overall, this shows that
\begin{align*}
\min_{\psi_\ell\in\XX_\ell} \norm{\psi-\psi_\ell}{H^{-1/2}(\Gamma)} 
\le \norm{\psi-\psi_\varepsilon}{H^{-1/2}(\Gamma)} + \norm{(1-J_\ell)\psi_\varepsilon}{H^{-1/2}(\Gamma)} 
\lesssim \varepsilon + \norm{h_\ell}{L^\infty(\Gamma)}^{1/2} |\psi_\varepsilon|_{H^{1/2}(\Gamma)}.
\end{align*}

Since $\lim_{\ell\to \infty}\norm{h_\ell}{L^\infty(\Gamma)}=0$ and $\varepsilon$ was arbitrary, this concludes the proof.

{\rm(b)} The latter observation allows to follow the ideas of~\cite{helmholtz} and to show that the adaptive algorithm yields convergence provided that the sesquilinear form $\sprod{\mathfrak{V}\,\cdot}{\cdot}$ is only elliptic up to some compact perturbation and that the continuous problem is well-posed. This includes, e.g., adaptive BEM for the Helmholtz equation; see~\cite[Section~6.9]{s}. For details, the reader is referred  to~\cite{helmholtz,helmholtzbem}.%
\end{remark}

\subsection{Application to BEM with piecewise polynomials on triangulations}\label{sec:standard applications} 
For $d=2,3$, we fix the reference simplex $T_{\rm ref}$ as the closed convex hull of the $d$ vertices $\{0,e_1,\dots,e_{d-1}\}$. 
The convex hull of any $d-1$ vertices is called \emph{facet}.
A set $\TT_\coarse$ of subsets of $\Gamma$ is called \emph{$\kappa$-shape regular triangulation} if the following properties {\rm (i)--(v)} are satisfied:
\begin{enumerate}[\rm (i)]
\item\label{item:gammaTs} $\TT_\coarse$ is a finite set of elements $T$ of the form $T=\gamma_T({\widehat{T}})$, where $\gamma_T:{  T_{\rm ref}}\to T$  is a bi-Lipschitz mapping whose Lipschitz constants are bounded from above by $\kappa$.
\item  $\TT_\coarse$ covers $\Gamma$, i.e., $\Gamma = \bigcup_{T\in\TT_\coarse}{T}$.
\item\label{item:conforming} There are no hanging nodes in the sense that the intersection $T\cap T'$ of any $T,T'\in\TT_\coarse$ with $T\neq T'$ is either empty or a common facet, i.e., $T\cap T'=\gamma_T(f)=\gamma_{T'}(f')$ for some facets $f$ and $f'$ of $T_{\rm ref}$.
\item\label{item:reference patch} The parametrizations of neighboring elements are compatible in the sense that for all nodes $z$ (i.e., images of the $\{0,e_1,\dots,e_{d-1}\}$ under an element map $\gamma_T$) , there exists an interval $\widetilde\pi_\coarse(z)$ for $d=2$ and a convex polygonal $\widetilde\pi_\coarse(z)$ for $d=3$ respectively as well as a bijective and  bi-Lipschitz continuous mapping $\gamma_z:\widetilde\pi_\coarse(z)\to \pi_\coarse(z)$ such that $\gamma_z^{-1}\circ\gamma_T$ is affine for all $T\in\Pi_\coarse(z)$.
\item If $d=2$, $\TT_\coarse$ is locally-quasi uniform in the sense that $\diam(T)\le \kappa\,\diam(T')$ for all $T,T'\in\TT_\coarse$ with $T\cap T'\neq \emptyset$.
\end{enumerate} 

Up to the fact that we allow $\gamma_T$ to be bi-Lipschitz  instead of $C^1$, this definition is slightly stronger than  \cite[Definition~2.4]{invest}. 
The property \eqref{item:conforming} stems from \cite[Assumption~4.3.25]{ss} and is stronger than the corresponding assumption \cite[Definition~2.4~(iii)]{invest}. 
Further,  \eqref{item:gammaTs} implies \cite[Definition~2.4~(v)]{invest}, i.e., for all $T\in\TT_\coarse$, there holds with the extremal eigenvalues $\lambda_{\rm min}(\cdot)$ and $\lambda_{\rm max}(\cdot)$ that
\begin{align}\label{eq:Gramian}
\sup_{t\in T_{\rm ref}}\Big(\frac{\diam(T)^2}{\lambda_{\rm min}(D\gamma_T^\top(t) D\gamma_T(t))}+\frac{\lambda_{\rm max}(D\gamma_T^\top (t) D\gamma_T(t))}{\diam(T)^2}\Big)\lesssim 1;
\end{align}
see, e.g.,  \cite[(3.26)--(3.27)]{faerconv} or \cite[Lemma~5.2.1]{diss}.

Let $\TT_0$ be a $\kappa_0$-shape regular triangulation.
For $d=2$, we define $\refine(\cdot)$ as in \cite{eps65} via 1D-bisection in the parameter domain. 
For $d=3$, we define $\refine(\cdot)$ as in \cite{stevenson08} via newest vertex bisection in the parameter domain.
In particular, all corresponding refinements $\TT_\coarse\in\T=\refine(\TT_0)$ are again $\kappa$-shape regular triangulations with some fixed $\kappa$ depending on $\kappa_0$.
We also note that the number of different $\widetilde\pi_\coarse(z)$ in \eqref{item:reference patch} is uniformly bounded, i.e., there exist only finitely many reference node patches.
Finally, let $p\in\N_0$ be a fixed polynomial order.
For each $\TT_\coarse$, we associate the space of (transformed) piecewise polynomials 
\begin{align}
\XX_\coarse:=\mathcal{P}^p(\TT_\coarse):= \set{\Psi_\coarse\in L^2(\Gamma)}{\Psi_\coarse\circ\gamma_T \text{ is a polynomial of degree }p \text{ for all }T\in\TT_\coarse }.
\end{align}
For this concrete setting, we already pointed out that \cite{part1} in combination with \cite{invest} proved linear convergence \eqref{eq:linear bem} at optimal rate \eqref{eq:optimal bem} if $\mathfrak{P}=-\Delta$ is the Laplace operator.
The following theorem generalizes this result to arbitrary $\mathfrak{P}$ as in Section~\ref{subsec:model problem bem}. 
\begin{theorem}\label{thm:concrete main}
Piecewise polynomials on $\kappa$-shape regular triangulations satisfy the abstract properties \eqref{M:patch bem}--\eqref{M:semi bem}, \eqref{R:sons bem}--\eqref{R:overlay bem}, and \eqref{S:inverse bem}--\eqref{S:stab bem}, where the constants depend only on the dimension $D$, the regularity constant $\kappa$, the initial mesh $\TT_\coarse$, and  the polynomial order $p$.
By Theorem~\ref{thm:abstract bem}, this implies reliability \eqref{eq:reliable bem} of the error estimator and linear convergence \eqref{eq:linear bem} at optimal rate \eqref{eq:optimal bem} for the adaptive strategy from Algorithm~\ref{alg:bem algorithm}.
\end{theorem}
\begin{proof}
The elementary mesh properties \eqref{M:patch bem}--\eqref{M:shape bem} are verified in \cite[Section~2.3]{invest}.
\eqref{M:cent bem} is stated in \cite[Section~4.1]{invest}.
\eqref{M:semi bem} follows from Proposition~\ref{prop:semi estimate} together with the fact that there are only finitely many reference node patches.

For $d=2$, the son estimate \eqref{R:sons bem} is clearly satisfied with $C_{\rm son}=2$.
For $d=3$, it is well-known that NVB satisfies \eqref{R:sons bem} with $C_{\rm son}=4$.
Further, \eqref{R:sons bem} holds true by definition.
Reduction of sons \eqref{R:reduction bem} is obviously satisfied in the parameter domain, i.e., $|\gamma_T^{-1}(T')| \le \,|\gamma_T^{-1}(T)|/2$ for all $T'\in\TT_\fine\in\refine(\TT_\coarse)$ with $T'\subsetneqq T$. 
Since $\gamma_T$ is bi-Lipschitz, this property transfers to the physical domain, i.e., $|\gamma_T^{-1}(T')| \le \,\ro{son}|\gamma_T^{-1}(T)|$, where $0<\ro{son}<1$ depends only on $\kappa$; see, e.g., \cite[Section~4.5.3]{diss} for details.
For $d=2$, \eqref{R:closure bem}--\eqref{R:overlay bem} are found in \cite[Theorem~2.3]{eps65}.
For $d=3$,  the closure estimate \eqref{R:closure bem} is  proved in \cite[Theorem~2.4]{bdd},  \cite[Theorem~6.1]{stevenson}, or \cite[Theorem~2]{kpp13}, where the latter result avoids any additional assumption on $\TT_0$.
The overlay property is proved in \cite[Proof of Lemma~5.2]{stevenson} or \cite[Section~2.2]{ckns}.

The inverse inequality  \eqref{S:inverse bem} for piecewise polynomials on the boundary is proved, e.g., in \cite[Lemma~A.1]{invest}. 
Nestedness \eqref{S:nestedness bem} is trivially satisfied.
Also \eqref{S:local bem} is trivially satisfied with $\q{loc},\q{proj}=0$.
Clearly, \eqref{S:unity bem} holds with $(\Psi_{\coarse,T,j})_{j'}:=0$ for $j'\neq j$ and $(\Psi_{\coarse,T,j})_j:=\chi_T$, where $\chi_T$ denotes the indicator function on $T$.
Finally, for $\SS\subseteq\TT_\coarse\in\T$, we define with the elementwise $L^2(T)$-orthogonal projection $P_{\coarse,T}:L^2(T)^D\to \set{\Psi_\coarse|_T}{\Psi_\coarse\in\XX_\coarse}$
\begin{align}
J_{\coarse,\SS}:L^2(\Gamma)\to\XX_\coarse,\quad \psi\mapsto J_{\coarse,\SS}:=\begin{cases}P_{\coarse,T}\psi\quad&\text{on all }T\in\SS, \\ 0 \quad&\text{on all }T\in\TT_\coarse\setminus\SS.\end{cases}
\end{align}
This definition immediately yields \eqref{S:proj bem}--\eqref{S:stab bem} with $\q{sz}=0$.
\end{proof}

\begin{remark}\label{rem:concrete main}
We mention that Theorem~\ref{thm:concrete main} is also  valid if $d=2$ and $\XX_\coarse$ is chosen as set of (transformed) {splines} which are piecewise polynomials with certain differentiability conditions at the break points.
The required properties are (implicitly) verified in \cite{resigabem}.
As in \cite{igafem} (resp.\ \cite{tigafem}), where we have verified the abstract FEM framework of \cite{igafem} for IGAFEM with hierarchical splines~\cite{juttler} and the mesh-refinement from \cite{igafem} (resp.\ T-splines with the mesh-refinement from~\cite{morgensternT1}), 
the verification of the present abstract BEM framework for 3D IGABEM will be addressed in the future work~\cite{riga}.
\end{remark}


\section{Proof of Theorem~\ref{thm:abstract bem}}\label{sec:proof abstract bem}
In the following   subsections, we prove Theorem~\ref{thm:abstract bem}.
Reliability \eqref{eq:reliable bem}  is  treated explicitly in  Section~\ref{subsec:reliability bem}.
It follows immediately from an auxiliary result on the localization of the Sobolev--Slobodeckij norm which is investigated in Section~\ref{subsec:localization}.
To prove Theorem~\ref{thm:abstract bem}~\eqref{item:qabstract linear convergence bem}--\eqref{item:qabstract optimal convergence bem}, 
we  verify the following abstract properties \eqref{item:stability}--\eqref{item:discrete reliability} for the error estimator.
Together with \eqref{R:sons bem}, the closure estimate \eqref{R:closure bem}, and the overlay property \eqref{R:overlay bem}, these already imply linear convergence of the estimator at optimal algebraic rate; see   \cite{axioms}.

There exist $\const{\rho}, \const{qo}$, $\const{ref}$, $\const{drel}$, $\const{son}$, $\const{clos}$, $\const{over}\ge 1$, 
and $0\le\ro{red},\varepsilon_{\rm qo},\varepsilon_{\rm drel}<1$ such that there hold:
\begin{enumerate}[(1)]
\renewcommand{\theenumi}{E\arabic{enumi}}
\bf\item\rm\label{item:stability}\textbf{Stability on non-refined elements:} For all  $\TT_\coarse\in\T$ and $\TT_{\fine}\in\refine(\TT_\coarse)$, 
it holds that
\begin{align*}
|\eta_{\fine}(\TT_\coarse\cap\TT_{\fine})-\eta_\coarse(\TT_\coarse\cap\TT_{\fine})|\le\varrho_{\coarse,\fine}:=\const{\rho}\norm{\Phi_\fine-\Phi_\coarse}{H^{-1/2}(\Gamma)}.
\end{align*}
\bf\item\rm\label{item:reduction}\textbf{Reduction on refined elements:} For all  $\TT_\coarse\in\T$ and $\TT_{\fine}\in\refine(\TT_\coarse)$, 
it holds that
\begin{align*}
\eta_{\fine}(\TT_{\fine}\setminus\TT_\coarse)^2\le\ro{red}\eta_\coarse( \TT_\coarse\setminus\TT_{\fine})^2+\varrho_{\coarse,\fine}^2.
\end{align*}
\bf\item\rm\label{item:orthogonality} \textbf{General quasi-orthogonality:} It holds that
\begin{align*}
0\le\varepsilon_{\rm qo}<\sup_{\delta>0}\frac{1-(1+\delta)(1-(1-\ro{red})\theta)}{2+\delta^{-1}},
\end{align*}
and  for all $\ell,N\in\N_0$, the sequence $(\TT_\ell)_{\ell\in\N_0}$ from Algorithm~\ref{alg:bem algorithm} satisfies that
\begin{align*}
\sum_{j=\ell}^{\ell+N}(\varrho_{j,j+1}^2-\varepsilon_{\rm qo}\eta_j^2)\le\const{qo} \eta_\ell^2.
\end{align*}
\bf\item\rm \label{item:discrete reliability}\textbf{Discrete reliability:} For all $\TT_\coarse\in\T$ and all  $\TT_{\fine}\in\refine(\TT_\coarse)$, 
 there exists $\TT_\coarse\setminus\TT_{\fine}\subseteq\RR_{\coarse,\fine}\subseteq\TT_\coarse$ with
 $\#\RR_{\coarse,\fine}\le\const{ref}(\#\TT_{\fine}-\#\TT_\coarse)$ such that
\begin{align*}
\varrho_{\coarse,\fine}^2\le\varepsilon_{\rm drel}\eta_\coarse^2+\const{drel}^2\eta_\coarse(\RR_{\coarse,\fine})^2.
\end{align*}
\end{enumerate}


\subsection{Localization of the Sobolev--Slobodeckij norm}
\label{subsec:localization}
Let $\TT_\coarse\in\T$. 
In contrast to the integer-case, for $\sigma\in(0,1)$, the norm $\norm{\cdot}{H^\sigma(\Gamma)}$ is not additive in the sense that 
\begin{align*}
\norm{v}{H^\sigma(\Gamma)}^2\simeq \sum_{T\in\TT_\coarse} \norm{v}{H^\sigma(T)}^2\quad\text{for all }v\in H^\sigma(\Gamma)^D. 
\end{align*}
Although the upper bound "$\lesssim$" is in general false  (see \cite[Section~3]{cf01}),  the lower bound "$\gtrsim$" can be proved elementarily for arbitrary $v\in H^\sigma(\Gamma)^D$.
\begin{proposition}\label{prop:easy faermann}
Let $0<\sigma<1$ and $\TT_\bullet\in\T$.
Then,  \eqref{M:patch bem} implies the  existence of  a constant $C_{\rm split}'>0$ such that for any $v\in H^\sigma(\Gamma)^D$, there holds that 
\begin{align}
\sum_{T\in\TT_\bullet}\sum_{T'\in \Pi_\bullet(T)}|v|_{H^\sigma(T\cup T')}^2\le C_{\rm split}'\seminorm{v}{H^\sigma(\Gamma)}^2. 
\end{align}
The constant $C_{\rm split}'$ depends only on the constant from \eqref{M:patch bem}.\hfill$\square$
\end{proposition}
\begin{proof}
With the abbreviation
\begin{align}\label{eq:V notation}
V(x,y):=\frac{|v(x)-v(y)|^2}{|x-y|^{d-1+2\sigma}}\quad\text{for all }x,y\in\Gamma\text{ with }x\neq y,
\end{align}
\eqref{M:patch bem} shows that 
\begin{align*}
&\sum_{T\in\TT_\bullet}\sum_{T'\in \Pi_\bullet(T)}|v|_{H^\sigma(T\cup T')}^2= \sum_{T\in\TT_\bullet}\sum_{T'\in \Pi_\bullet(T)} \Big(|v|_{H^\sigma(T)}^2+2\int_T\int_{T'} V(x,y) dx dy+|v|_{H^\sigma(T')}^2\Big)\\
&\quad=2\sum_{T\in\TT_\bullet}\sum_{T'\in \Pi_\bullet(T)} \Big(\int_T\int_{T} V(x,y) dx dy+\int_T\int_{T'} V(x,y) dx dy\Big)
\le 2(\const{patch}+1) |v|_{H^\sigma(\Gamma)}^2.
\end{align*}
This concludes the proof.
\end{proof}
However, if one replaces the elements $T$ by some overlapping patches,  then also the converse inequality is satisfied for functions $v\in H^\sigma(\Gamma)^D$ which are $L^2$-orthogonal to the ansatz space $\XX_\coarse$.

\begin{proposition}\label{prop:faermann}
Let $0<\sigma<1$ and $\TT_\bullet\in\T$.
Then,  \eqref{M:patch bem}--\eqref{M:cent bem} and \eqref{S:unity bem} imply the  existence of  a constant $C_{\rm split}>0$ such that for any $v\in H^\sigma(\Gamma)^D$ which satisfies  that $\sprod{v}{(\Psi_{\bullet,T,j})_j}_{L^2(\Gamma)}=0$ for all $T\in\TT_\bullet$ and all $j\in\{1,\dots,\D\}$, where $\Psi_{\bullet,T,j}$ are the functions from \eqref{S:unity bem}, it holds that 
\begin{align}\label{eq:faermann}
\norm{v}{H^\sigma(\Gamma)}^2\le C_{\rm split}\sum_{T\in\TT_\bullet}\sum_{T'\in \Pi_\bullet(T)}|v|_{H^\sigma(T\cup T')}^2. 
\end{align}
The constant $C_{\rm split}$ depends only on 
 the dimension $d,\sigma$,  $\Gamma$, and  the constants from \eqref{M:patch bem}--\eqref{M:cent bem} and \eqref{S:unity bem}.
\end{proposition}

With this result, one can immediately construct a reliable and efficient  error estimator, namely the so-called \textit{Faermann estimator}; see Remark~\ref{rem:faermann}.
For $d=2$, the result of the proposition goes back to \cite{faer2}, where $\XX_\coarse$ is chosen as space of splines transformed via the arclength parametrization $\gamma:[a,b]\to \Gamma$ onto the one-dimensional boundary.
In the recent own works \cite{fgp}, we generalized the assertion to rational splines, where we could also drop the restriction that $\gamma$ is the arclength parametrization.
For $d=3$, \cite{faer3}  proved the result for discrete spaces which contain certain (transformed) polynomials of degree $p\in\{0,1,5,6\}$ on a curvilinear triangulation of $\Gamma$. 
Our proof of Proposition~\ref{prop:faermann} is  inspired by \cite{faer3}.
The key ingredient is the assumption \eqref{S:unity bem} which is exploited in Lemma~\ref{lem:L2 to Hhalf}.
Before proving Proposition~\ref{prop:faermann}, we provide an easy corollary which is the key ingredient for the proof of reliability \eqref{eq:reliable bem}.

\begin{corollary}\label{cor:global poincare}
Let $\TT_\bullet\in\T$. 
Then,  \eqref{M:patch bem}--\eqref{M:semi bem}  and \eqref{S:unity bem} imply the  existence of  a constant $C_{\rm rel}'>0$ such that for any $v\in H^1(\Gamma)^D$ which satisfies  that $\sprod{v}{\Psi_{\bullet,T,j}}_{L^2(\Gamma)}=0$ for all $T\in\TT_\bullet$ and all $j\in\{1,\dots,\D\}$, where $\Psi_{\bullet,T,j}$ are the functions from \eqref{S:unity bem}, it holds that 
\begin{align}\label{eq:global poincare}
\norm{v}{H^{1/2}(\Gamma)}\le C_{\rm rel}'\norm{h_\bullet^{1/2}\nabla_\Gamma v}{L^2(\Gamma)}. 
\end{align}
The constant $C_{\rm rel}'$ depends only on 
 the dimension $d$, $\Gamma$, as well as the constants from \eqref{M:patch bem}--\eqref{M:semi bem}  and \eqref{S:unity bem}.\hfill$\square$
\end{corollary}

To prove Proposition~\ref{prop:faermann}, we start with the following basic estimate, which is  proved in \cite[Lemma~8.2.4]{hackbusch} or in \cite[Lemma~5.3.1]{diss}.

\begin{lemma}\label{lem:Clambda}
For all $\lambda>0$, there is a constant $C(\lambda)>0$ such that for all  $x\in \R^d$  and all $\varepsilon>0$, there holds that
\begin{align}
\int_{\Gamma\setminus B_{\varepsilon}(x)} |x-y|^{-d+1-\lambda} dy
 \le C({\lambda}) \varepsilon ^{-\lambda}.
\end{align}
 The constant $C(\lambda)$ depends only on the parameter $\lambda$, the dimension $d$, and $\Gamma$. \hfill$\square$
\end{lemma}

The following lemma is the first step towards the localization of the norm $\norm{v}{H^\sigma(\Gamma)}$ for certain functions $v\in H^\sigma(\Gamma)^D$.
In \cite[Lemma~3.1]{faer3}, this result is stated for triangular meshes.
The elementary proof extends to our situation; see also \cite[Lemma~5.3.2]{diss} for details.
\begin{lemma}\label{lem:first split}
Let  $0<\sigma<1$ and  $\TT_\bullet\in\T$.
Then, \eqref{M:cent bem} implies the existence of a constant $C>0$ such that for all $v\in H^\sigma(\Gamma)^D$, it holds that 
\begin{align}\label{eq:first split}
\norm{v}{H^\sigma(\Gamma)}^2\le 
\sum_{T\in\TT_\bullet}\sum_{T'\in\Pi_\bullet(T)}|v|_{H^\sigma(T\cup T')}^2 
+C\sum_{T\in\TT_\bullet} 
\diam(T)^{-2\sigma}\norm{v}{L^2(T)}^2.
\end{align}
The constant $C$ depends only on 
 the dimension $d$, $\sigma$, $\Gamma$, and the constant from \eqref{M:cent bem}.\hfill$\square$
\end{lemma}

It remains to control the second summand in \eqref{eq:first split}. 
To this end, we need the following elementary Poincar\'e type inequality of \cite[Lemma~2.5]{faer2}.
\begin{lemma}\label{lem:Poincare}
For any  $\sigma\in(0,1)$ and any measurable $\omega\subseteq \Gamma$,  there holds for all $v\in H^\sigma(\omega)$ that
\begin{align}\label{eq:SS poincare}
\norm{v}{L^2(\omega)}^2\le \frac{\diam(\omega)^{d-1+2\sigma}}{2|\omega|}|v|_{H^\sigma(\omega)}^2+\frac{1}{|\omega|}\left|\int_\omega v(x) dx \right|^2.
\end{align}\hfill$\square$
\end{lemma}
We start to estimate the  second summand in \eqref{eq:first split}.
\begin{lemma}\label{lem:L2 to Hhalf}
Let   $\sigma\in (0,1)$,  $\TT_\bullet\in\T$ and $T\in\TT_\bullet$.
Then, \eqref{M:patch bem}--\eqref{M:shape bem} and \eqref{S:unity bem} imply the existence of  a constant $C>0$ such that for all $v\in H^\sigma(\Gamma)^D$ with $\sprod{v_j}{\Psi_{\bullet,T,j}}_{L^2(\Gamma)}=0$ 
for all $j\in\{1,\dots,\D\}$, 
where $\Psi_{\bullet,T,j}$ are the functions from \eqref{S:unity bem}, it holds that
\begin{align}\label{eq:L2 to Hhalf}
\norm{h_\bullet^{-\sigma} v}{L^2(T)}
\le C |v|_{H^\sigma(\pi_\bullet^{\q{supp}}(T))},
\end{align}
where $\q{supp}$ is the constant from \eqref{S:unity bem}.
The constant $C$ depends only on 
the dimension $d$,  $\sigma$, $\Gamma$,  and the constants from  \eqref{M:patch bem}--\eqref{M:shape bem} and \eqref{S:unity bem}.
\end{lemma}
\begin{proof}
We prove \eqref{eq:L2 to Hhalf} for each component $v_j$ of $v$, where $j\in\{1,\dots,D\}$.
Then, squaring and summing up all components, we conclude the proof.
\eqref{S:unity bem} and Lemma~\ref{lem:Poincare}   show that 
\begin{align}\label{eq:T to psiT}
\begin{split}
&\norm{v_j}{L^2(T)}^2\le \norm{v_j}{L^2(\supp( \Psi_{\bullet,T,j}))}^2 \\
&\le \frac{\diam (\supp( \Psi_{\bullet,T,j}))^{d-1+2\sigma}}{2|\supp(\Psi_{\bullet,T,j})|}|v_j|^2_{H^\sigma(\supp( \Psi_{\bullet,T,j}))}+\frac{1}{|\supp( \Psi_{\bullet,T,j})|} \left|\int_{\supp (\Psi_{\bullet,T,j})} v_j(x) dx\right|^2.
\end{split}
\end{align}
Now, we apply  the orthogonality  and \eqref{S:unity bem} to get for the second summand that 
\begin{align*}
&\frac{1}{|\supp(\Psi_{\bullet,T,j})|}\left|\int_{\supp(\Psi_{\bullet,T,j})} v_j(x) dx\right|^2 =\frac{1}{|\supp(\Psi_{\coarse,T,j})|} \left|\int_{\supp(\Psi_{\coarse,T,j})} \overline v_j(x) (1-\Psi_{\coarse,T,j}(x)) dx\right|^2\\
&\quad\le \frac{1}{|\supp(\Psi_{\coarse,T,j})|} \norm{v_j}{L^2(\supp(\Psi_{\coarse,T,j}))}^2 \norm{1-(\Psi_{\coarse,T,j})_j}{L^2(\supp(\Psi_{\coarse,T,j}))}^2\le \ro{unity}^2 \norm{v_j}{L^2(\supp(\Psi_{\coarse,T,j})}^2.
\end{align*}
Inserting this in \eqref{eq:T to psiT} gives that 
\begin{align}\label{eq:L2 to Hhalf nearly}
(1-\ro{unity}^2)\norm{v_j}{L^2(\supp(\Psi_{\coarse,T,j}))}^2\le \frac{\diam(\supp(\Psi_{\coarse,T,j}))^{d-1+2\sigma}}{2|\supp(\Psi_{\coarse,T,j})|} |v_j|_{H^\sigma(\supp(\Psi_{\coarse,T,j}))}^2.
\end{align}
With \eqref{S:unity bem} and \eqref{M:patch bem}--\eqref{M:shape bem}, we see that
$\diam(\supp(\Psi_{\coarse,T,j}))\le \diam(\pi_\bullet^{\q{supp}}(T))\lesssim \diam(T)\simeq h_T$.
Further, \eqref{S:unity bem} 
implies that $|\supp(\Psi_{\coarse,T,j})|\ge  |T|= h_T^{d-1}$.
Inserting this in  \eqref{eq:L2 to Hhalf nearly} and using again \eqref{S:unity bem}, we derive that
\begin{align*}
\norm{v_j}{L^2(T)}^2\le\norm{v_j}{L^2(\supp(\Psi_{\coarse,T,j}))}^2\lesssim h_T^{2\sigma} |v_j|_{H^\sigma(\supp(\Psi_{\coarse,T,j}))}^2\le  h_T^{2\sigma} |v_j|_{H^\sigma(\pi_\bullet^{\q{supp}}))}^2.
\end{align*}
Altogether, this concludes the proof.
\end{proof}

The following lemma allows us to further estimate the term $|v|_{H^\sigma(\pi_\bullet^{\q{supp}}(T))}$ of \eqref{eq:L2 to Hhalf}.
\begin{lemma}\label{lem:patch to elements}
Let  $q\in\N_0$ and  $\TT_\bullet\in\T$.
Then,  \eqref{M:patch bem}--\eqref{M:cent bem}   imply the existence of a constant $C(q)>0$ such that for all $v\in H^\sigma(\Gamma)^D$ and all $T\in\TT_\bullet$ there holds that
\begin{align}\label{eq:patch2patch}
|v|_{H^\sigma(\pi_\bullet^q(T))}^2\le C(q) \sum_{T',T''\in \Pi_\bullet^q(T)\atop T'\cap T''\neq\emptyset} |v|_{H^\sigma( T'\cup T'')}^2.
\end{align}
The constant depends only on the dimension $d, \sigma,q,$
 and  the constants from \eqref{M:patch bem}--\eqref{M:cent bem}.
\end{lemma}
\begin{proof}
Without loss of generality, we may assume that $D=1$.
We prove the assertion in two steps.

{\bf Step 1:}
Let   $T_{0},T_{1},\dots,T_{m}$  be a \textit{chain} of elements  in $\Pi_\bullet^q(T)$ with $T_{i}\cap T_{j}=\emptyset$ for $|i-j|>1$ and  $T_{i}\cap T_{j}\neq \emptyset$ if $|i-j|=1$, where $1\le m \le q$.
We set $T_i^j:=\bigcup_{k=i}^j T_{\ell}$ for $i\le j$ and prove by induction on $m$ that there exists a constant $C_1(m)>0$ which depends only on $d,\sigma,q,m,$ and \eqref{M:locuni bem}--\eqref{M:cent bem},  such that
\begin{align}\label{eq:chain estimate}
|v|_{H^\sigma(T_0^m)}^2\le C_1(m) \sum_{i=0}^{m-1} |v|_{H^\sigma(T_{i}\cup T_{i+1})}^2.
\end{align}
For $m=1$, \eqref{eq:chain estimate} with $C_1(1)=1$ even holds with equality.
Thus,  the induction hypothesis  reads:
For all $1\le m-1< q$ and for any chain $T_{0},\dots,T_{m-1}$ of elements  in $\Pi_\bullet^q(T)$, it holds that 
\begin{align}\label{eq:chain estimate hypothesis}
|v|_{H^\sigma(T_0^{m-1})}^2 \le C_1(m-1) \sum_{i=0}^{m-2} |v|_{H^\sigma(T_{i}\cup T_{i+1})}^2.
\end{align}
Let $T_{m}\in\Pi_\bullet^q(T)$ with $T_{m}\cap T_{i}=\emptyset$ for $i\le m-2$ and $T_{m}\cap T_{i}\neq \emptyset$ for $i=m-1$.
For all $x,y\in\Gamma,x\neq y$, we abbreviate $V(x,y):=\frac{|v(x)-v(y)|^2}{|x-y|^{d-1+2\sigma}}$.
The definition \eqref{eq:SS-norm} of the Sobolev-Slobodeckij seminorm shows that 
\begin{align*}
|v|_{H^\sigma({T_0^m})}^2&=\int_{T_0^m}\int_{T_0^m}V(x,y) dx dy\\
&=\int_{T_0^{m-1}}\int_{T_0^{m-1}}V(x,y) dx dy+\int_{T_m}\int_{T_m} V(x,y) dx dy+2\int_{T_m}\int_{T_0^{m-1}}V(x,y) dx dy\\
&=|v|_{H^\sigma(T_0^{m-1})}^2+|v|_{H^\sigma(T_{m})}^2+2\int_{T_m}\int_{T_0^{m-2}} V(x,y) dx dy+2\int_{T_m}\int_{T_{m-1}} V(x,y) dx dy\\
&\le |v|_{H^\sigma(T_0^{m-1})}^2+|v|_{H^{\sigma}(T_{m-1}\cup T_m)}^2+2\int_{T_{m}}\int_{T_0^{m-2}}V(x,y) dx dy.
\end{align*}
With the induction hypothesis \eqref{eq:chain estimate hypothesis}, it remains to estimate $\int_{T_m}\int_{T_0^{m-2}} V(x,y) dx dy$.
First, we note that for $x\in T_0^{m-2},y\in T_{m}, z\in T_{m-1}$, it holds that
\begin{align}\label{eq:U triangle1}
V(x,y)&=\frac{|v(x)-v(y)|^2}{|x-y|^{d-1+2\sigma}}\le 2\frac{|v(x)-v(z)|^2}{|x-y|^{d-1+2\sigma}}+2\frac{|v(z)-v(y)|^2}{|x-y|^{d-1+2\sigma}}.
\end{align}
Moreover, \eqref{M:cent bem} shows  that $|x-y|\ge \dist(T_m,\Gamma\setminus\pi_\bullet(T_m))\gtrsim \diam({T_m})$.
Since $x,y,z\in T_0^m$, \eqref{M:locuni bem} shows $\max\{|x-z|,|y-z|\}\lesssim \diam({T_m})$.
Hence, we can proceed the estimate of \eqref{eq:U triangle1} 
\begin{align*}
V(x,y)\lesssim V(x,z)+V(z,y).
\end{align*}
This implies that
\begin{align*}
&\int_{T_m} \int_{T_0^{m-2}}V(x,y) dx dy=\frac{1}{|T_{m-1}|} \int_{T_{m-1}}\int_{T_m}\int_{T_0^{m-2}}V(x,y) dx dy dz\\
&\quad\lesssim \frac{1}{|T_{m-1}|} \int_{T_{m-1}}\int_{T_m}\int_{T_0^{m-2}} V(x,z) +V(y,z) dx dy dz\\
&\quad=\frac{1}{|T_{m-1}|} \left(\int_{T_{m-1}}\int_{T_0^{m-2}} |T_{m}| V(x,z) dx dz +\int_{T_{m-1}}\int_{T_{m-1}} |T_0^{m-2}| V(y,z) dy dz\right)\\
&\quad\le \frac{\max\{|T_{m}|,|T_0^{m-2}|\}}{|T_{m-1}|} \Big(|v|_{H^\sigma(T_0^{m-1})}^2+|v|_{H^\sigma(T_{m-1}\cup T_m)}^2\Big).
\end{align*}
Note that $ \max\{|T_{m}|,|T_0^{m-2}|\}/{|T_{m-1}|} \lesssim 1$ by \eqref{M:locuni bem}--\eqref{M:shape bem}.
Together with  the induction hypothesis \eqref{eq:chain estimate hypothesis}, this concludes the induction step.

{\bf Step 2:}
We come to the assertion itself.
By definition, we have that
\begin{align*}
|v|_{H^{1/2}(\pi_\bullet^q(T))}^2=\sum_{\widetilde T',\widetilde T''\in\Pi_\bullet^q(T)}\int_{\widetilde T'}\int_{\widetilde T''} V(x,y) dx dy.
\end{align*}
Let  $\widetilde T', \widetilde T''\in\Pi_\bullet^q(T)$.
First, we suppose that    $\widetilde T'\neq \widetilde T''=\emptyset$.
Then, there exists a chain as in Step~1 with $\widetilde T'=T_0$ and $\widetilde T''=T_m$.
Step 1 proves that 
\begin{align*}
\int_{\widetilde T'}\int_{\widetilde T''}V(x,y) dx dy\le |v|_{H^\sigma(T_0^m)}^2\lesssim \sum_{  T', T''\in \Pi_\bullet^q(T)\atop  T'\cap  T''\neq \emptyset} |v|_{H^\sigma( T'\cup T'')}^2.
\end{align*}
If $\widetilde T'= \widetilde T''$, the same estimate holds true. Since the number of $\widetilde T',\widetilde T''\in\Pi_\bullet^q(T)$ is uniformly bounded by a constant, which depends only on the constant of \eqref{M:patch bem} and $q$, this  estimate concludes the proof.
\end{proof}

With the property~\eqref{M:semi bem}, one immediately derives the following Poincar\'e inequality.
\begin{proposition}\label{prop:poincare}
Let    $\TT_\bullet\in\T$ and $T\in\TT_\bullet$.
Then, \eqref{M:patch bem}--\eqref{M:semi bem}  and \eqref{S:unity bem} imply the existence of  a constant $\const{poinc}>0$ such that for all $v\in H^1(\Gamma)^D$ which satisfy  that $\sprod{v}{\Psi_{\bullet,T,j}}_{L^2(\Gamma)}=0$  for all $j\in\{1,\dots,\D\}$, where $\Psi_{\bullet,T,j}$ are the functions from \eqref{S:unity bem}, it holds that
\begin{align}\label{eq:poincare}
\norm{h_\bullet^{-1}v}{L^2(T)}
\le \const{poinc}| v|_{H^1\big(\pi^{\q{supp}+1}_\bullet(T)\big)},
\end{align}
where $\q{supp}$ is the constant from \eqref{S:unity bem}.
The constant $\const{poinc}$ depends only on the dimension $d$, $\Gamma$, and  the constants from \eqref{M:patch bem}--\eqref{M:semi bem}  and \eqref{S:unity bem}.
\end{proposition}
\begin{proof}
We  apply Lemma~\ref{lem:L2 to Hhalf} and  Lemma~\ref{lem:patch to elements} to see that

\begin{align*}
\norm{h_\coarse^{-1/2}v}{L^2(T)}^2\lesssim|v|_{H^{1/2}(\pi_\bullet^{\q{supp}}(T))}^2\lesssim \sum_{T',T''\in \Pi_\bullet^{\q{supp}}(T)\atop T'\cap T''\neq\emptyset} |v|_{H^{1/2}( T'\cup T'')}^2.
\end{align*}
For $T',T''\in\TT_\coarse$ with $T'\cap T''\neq\emptyset$, we fix some point $z(T',T'')\in T'\cap T''$.
With {\rm \eqref{M:semi bem}},
we continue our estimate
\begin{align*}
\norm{h_\coarse^{-1/2}v}{L^2(T)}^2&\lesssim|v|_{H^{1/2}(\pi_\bullet^{\q{supp}}(T))}^2\lesssim \sum_{T',T''\in \Pi_\bullet^{\q{supp}}(T)\atop T'\cap T''\neq\emptyset} |v|_{H^{1/2}( \pi_\bullet(z(T',T''))}^2\\
&\lesssim \sum_{T',T''\in \Pi_\bullet^{\q{supp}}(T)\atop T'\cap T''\neq\emptyset}\diam\big(\pi_\bullet(z(T',T'')\big) \norm{\nabla_\Gamma v}{L^2( \pi_\bullet(T'))}^2.
\end{align*}
\eqref{M:patch bem}--\eqref{M:shape bem} imply that  $h_T\simeq h_\bullet$ on $\pi^{\q{supp}+1}_\bullet(T)$, and that the last term of the latter estimate can be bounded from above (up to a multiplicative constant) by $ \norm{h_\bullet^{1/2}\nabla_\Gamma v}{L^2\big(\pi^{\q{supp}+1}_\bullet(T)\big)}^2$. 
 This concludes the proof.
\end{proof}

With  all the preparations, we can  finally prove  the main result of this section.
\begin{proof}[Proof of Proposition~\ref{prop:faermann}]
Together with \eqref{M:shape bem}, Lemma \ref{lem:first split}   proves that 
\begin{align*}
\norm{v}{H^\sigma(\Gamma)}^2\lesssim
\sum_{T\in\TT_\bullet}\sum_{T'\in \Pi_\bullet(T)}|v|_{H^\sigma(T\cup T')}^2
+\sum_{T\in\TT_\bullet} h_T^{-2\sigma} \norm{v}{L^2(T)}^2.
\end{align*}
It remains to estimate the second sum.
With Lemma \ref{lem:L2 to Hhalf} and Lemma \ref{lem:patch to elements}, we see that 
\begin{align}\label{eq:L2 to Hhalf nonlocal}
\sum_{T\in\TT_\bullet} h_T^{-2\sigma} \norm{v}{L^2(T)}^2\lesssim \sum_{T\in\TT_\bullet} |{v}|_{H^\sigma(\pi_\bullet^{\q{supp}}(T))}^2\lesssim \sum_{T\in\TT_\bullet}\sum_{T',T''\in \Pi_\bullet^{\q{supp}}(T)\atop T'\cap T''\neq\emptyset} |{v}|_{H^\sigma( T'\cup T'')}^2.
\end{align}
If $T\in\TT_\bullet$ and $T',T''\in\Pi_\bullet^{\q{supp}}(T)$ with $T'\cap T''\neq\emptyset$, then $T\in\Pi_\bullet^{\q{supp}}(T')$ and $T''\in\Pi_\bullet(T')$.
Plugging this into \eqref{eq:L2 to Hhalf nonlocal} shows that 
\begin{align*}
\sum_{T\in\TT_\bullet} h_T^{-2\sigma} \norm{v}{L^2(T)}^2\lesssim \sum_{T'\in\TT_\bullet}\sum_{T\in \Pi_\bullet^{\q{supp}}(T')}\sum_{T''\in\Pi_\bullet(T')}|{v}|_{H^\sigma( T'\cup T'')}^2,
\end{align*}
and $\#\Pi_\bullet^{\q{supp}}(T')\lesssim 1$ (see \eqref{M:patch bem}) concludes the proof.
\end{proof}

\subsection{Reliability~(\ref{eq:reliable bem})}
\label{subsec:reliability bem}
Let $\TT_\coarse\in\T$. 
Recall that $\mathfrak{V}:H^{-1/2}(\Gamma)^D\to H^{1/2}(\Gamma)^D$ is an isomorphism.
Due to Galerkin orthogonality \eqref{eq:galerkin bem},  Corollary~\ref{cor:global poincare} leads to
\begin{align}\label{eq:faermann1}
\norm{\phi-\Phi_\coarse}{H^{-1/2}(\Gamma)}&\simeq \norm{f-\mathfrak{V}\Phi_\coarse}{H^{1/2}(\Gamma)}\lesssim\norm{h_\coarse^{1/2}\nabla_\Gamma(f-\mathfrak{V}\Phi_\coarse)}{L^2(\Gamma)}=\eta_\coarse.
\end{align}
\begin{remark}\label{rem:faermann}
Proposition~\ref{prop:easy faermann} and Proposition \ref{prop:faermann} show that 
\begin{align}\label{eq:faerest releff}
\norm{\phi-\Phi_\coarse}{H^{-1/2}(\Gamma)}^2\simeq\norm{f-\mathfrak{V}\Phi_\coarse}{H^{1/2}(\Gamma)}^2\simeq \sum_{T\in\TT_\bullet}\sum_{T'\in \Pi_\bullet(T)}|f-\mathfrak{V}\Phi_\coarse|_{H^{1/2}(T\cup T')}^2.
\end{align}
This is even true for arbitrary $f\in H^{1/2}(\Gamma)^D$ without the additional restriction $f\in H^1(\Gamma)^D$.
In particular, 
\begin{align}
\digamma_{\hspace{-1mm} \coarse}(T)^2:=\sum_{T'\in \Pi_\bullet(T)}|f-\mathfrak{V}\Phi_\coarse|_{H^{1/2}(T\cup T')}^2\quad\text{for all }T\in\TT_\coarse
\end{align}
  provides a local error indicator.
The corresponding error estimator $\digamma_{\hspace{-1mm} \coarse}$ is  often referred to as {\rm Faermann estimator}. 
In BEM, it is the only known estimator which is  reliable and efficient (without further assumptions as, e.g., the saturation assumption \cite[Section~1]{leite}).
Obviously, one could replace the residual estimator $\eta_\ell$ in Algorithm~\ref{alg:bem algorithm} by $\digamma_{\hspace{-1mm} \ell}$.
However, due to the lack of an $h$-weighting factor, it is unclear whether the reduction property \eqref{item:reduction} of Section~\ref{subsec:reliability bem} is satisfied.
\cite[Theorem~7]{faerconv} proves at least plain convergence of $\digamma_{\hspace{-1mm} \ell}$ even for $f\in H^{1/2}(\Gamma)^D$ if one uses piecewise constants on affine triangulations of $\Gamma$ as ansatz space.
The proof immediately extends to our current situation, where the assumptions \eqref{M:patch bem}--\eqref{M:semi bem}, \eqref{R:union bem}--\eqref{R:reduction bem}, and \eqref{S:inverse bem}--\eqref{S:nestedness bem} are employed.
The key ingredient is the construction of an equivalent mesh-size function $\widetilde h_\coarse\in L^\infty(\Gamma)$ which is contractive on each element which touches a refined element, i.e., there exists a uniform constant $0<\ro{ctr}<1$ such that 
\begin{align}
\widetilde h_\fine|_T \le \ro{ctr}\widetilde h_\coarse|_T\quad \text{for all } \TT_\fine\in\refine(\TT_\coarse) \text{ and all }T\in\Pi_\coarse(\TT_\coarse\setminus\TT_\fine).
\end{align}
The existence of such a mesh-size function is proved in \cite[Section~8.7]{axioms} for shape-regular triangular meshes.
The proof works  verbatim for the present setting.
\end{remark}

\subsection{Convergence of $\boldsymbol{\norm{\Phi_{\ell+1}-\Phi_{\ell}}{H^{-1/2}(\Gamma)}}$}\label{subsec:perturbations bem}
Nestedness~\eqref{S:nestedness bem} ensures that $\XX_\infty:=\overline{\bigcup_{\ell\in\N_0}\XX_\ell}$ is a closed subspace of $H^{-1/2}(\Gamma)^D$ and hence admits a unique Galerkin solution $\Phi_\infty\in\XX_\infty$. Note that $\Phi_\ell$ is also a Galerkin approximation of $\Phi_\infty$. Hence, the C\'ea lemma~\eqref{eq:cea bem} with $\phi$ replaced by $\Phi_\infty$  
 proves that $\norm{\Phi_\infty-\Phi_\ell}{H^{-1/2}(\Gamma)}\to0$ as $\ell\to\infty$. 
 In particular, we obtain that  $\lim_{\ell\to\infty}\norm{\Phi_{\ell+1}-\Phi_{\ell}}{H^{-1/2}(\Gamma)}=0$.

\subsection{An inverse inequality for $\mathfrak{V}$}
In Proposition~\ref{prop:invest for V}, we establish an inverse inequality for the single-layer operator $\mathfrak{V}$.
Throughout this section, the ellipticity of $\mathfrak{V}$ is not exploited  (and we can drop this assumption here).
For the Laplace operator $\mathfrak{P}=-\Delta$, such an estimate was already proved in \cite[Theorem~3.1]{fkmp} for 
shape-regular triangulations of a polyhedral boundary $\Gamma$.
Independently, \cite{gantumur} derived a similar result for globally smooth $\Gamma$ and arbitrary self-adjoint and elliptic boundary integral operators. 
In \cite[Theorem~3.1]{invest},  \cite[Theorem~3.1]{fkmp} is generalized to piecewise polynomial ansatz functions on shape-regular curvilinear triangulations. 
In particular, our Proposition~\ref{prop:invest for V} does not only extend these results to arbitrary general meshes as in Section~\ref{subsec:boundary discrete bem}, but  is also completely novel for, e.g., linear elasticity.
The proof follows the lines of \cite[Section~4]{invest}.
We start with the following lemma, which was proved in \cite[Theorem~4.1]{averaging} on shape-regular triangulations.
With Lemma~\ref{lem:Poincare}, the proof immediately extends to our situation; see also \cite[Lemma~5.3.11]{diss}.

\begin{lemma}\label{lem:cp41}
For $\TT_\coarse\in\T$, let $\PP^{0}(\TT_\coarse)^D\subset L^2(\Gamma)^D$ be the set of all functions whose $D$ components are $\TT_\coarse$-piecewise constant functions on $\Gamma$.
Let $P_\coarse:L^2(\Gamma)^D\to \PP^{0}(\TT_\coarse)^D$ be the corresponding $L^2$-projection.
Then, \eqref{M:patch bem} and \eqref{M:shape bem} imply for arbitrary $0<\sigma<1$ the existence of a constant $C>0$ such that
\begin{align}\label{eq:cp41}
\norm{(1-P_\coarse)\psi}{H^{-\sigma}(\Gamma)}\le C\norm{h_\coarse^\sigma\psi}{L^2(\Gamma)}\quad\text{for all }\psi\in L^2(\Gamma).
\end{align}
The constant $C$ depends only on the dimension $D$, the boundary $\Gamma$, $\sigma$,  and the constants from 
\eqref{M:shape bem}.\hfill$\square$
\end{lemma}

In contrast to \cite{invest}, we cannot use the Caccioppoli type inequality from \cite[Lemma~5.7.1]{morrey} which is only shown for the Poisson problem there.
Therefore, we prove the following generalization.
For an open set $O\subset\R^d$ and an arbitrary $u\in H^2(O)$, we abbreviate $\seminorm{u}{H^1(O)}:=\norm{\nabla u}{L^2(O)}$ and $\seminorm{u}{H^2(O)}:=(\sum_{i=1}^d\seminorm{\partial_i u}{H^1(O)}^2)^{1/2}$.

\begin{lemma}\label{lem:cacc}
Let $r>0$, $x\in\R^d$, and $u\in H^1(B_{2r}(x))^D$ be a weak solution of $\mathfrak{P}u=0$.
Then, $u|_{B_r(x)}\in C^\infty(B_r(x))^D$ and there exists a constant $C>0$ such that 
\begin{align}\label{eq:cacc}
|u|_{H^2(B_r(x))}\le C\big(\norm{u}{L^2(B_{2r}(x))} + \frac{1+r+r^2}{r}|u|_{H^1(B_{2r}(x))}\big).
\end{align}
The constant $C$ depends only on the dimensions $d,\D,$ and the coefficients of the partial differential operator $\mathfrak{P}$.
\end{lemma}
\begin{proof}
By \cite[Theorem~4.16]{mclean}, there holds  that $u|_{B_{3r/2}(x)}\in H^k(B_{3r/2}(x))^D$ for all $k\in\N_0$, and the Sobolev embedding theorem proves that $u|_{B_{3r/2}(x)}\in C^\infty(B_{3r/2}(x))^D$.
In particular, $u$ is a strong solution of $\mathfrak{P}u=0$ on ${B_{3r/2}(x)}$.
To prove \eqref{eq:cacc}, let $\lambda\in\R^{\D}$ be an arbitrary constant vector, and  define $\widetilde{u}:=u\circ\varphi$ with the affine bijection $\varphi:{B_{3/2}(0)}\to{B_{3r/2}(x)}$, $\varphi(\widetilde y)=r\widetilde y+x$ for $\widetilde y\in B_{3/2}(0)$.
Since the coefficients of  $\mathfrak{P}$ are constant and $u$ is a strong solution, there holds for all $\widetilde y\in{B_{3/2}(0)}$ with $y:=\varphi(\widetilde y)$ that
\begin{align}\label{eq:PDE on reference}
\begin{split}
-\sum_{i=1}^d\sum_{i'=1}^d \partial_i(A_{ii'}\partial_{i'}  (\widetilde{u}-\lambda))(\widetilde{y})&=-\sum_{i=1}^d\sum_{i'=1}^d \partial_{i}(A_{ii'}\partial_{i'} (u-\lambda))(y) \,r^2 \\
&= -r^2\Big(\sum_{i=1}^d b_i\partial_i (u-\lambda)(y)+ c (u-\lambda)(y)+c\,\lambda\Big).
\end{split}
\end{align}
We  define the right-hand side as $\widetilde f\in C^\infty({B_{3/2}(0)})$, i.e.,
\begin{align*}
\widetilde f(\widetilde y):= -r^2\Big(\sum_{i=1}^d b_i\partial_i (u-\lambda)(\varphi(\widetilde y))+ c (u-\lambda)(\varphi(\widetilde y))+c\,\lambda\Big).
\end{align*}
This shows that $\widetilde u-\lambda$ is a strong  (and thus weak) solution of a \textit{strongly elliptic} (see Section~\ref{subsec:model problem bem}) system of second-order PDEs with smooth coefficients and smooth right-hand side.
The application of \cite[Theorem~4.16]{mclean} yields the existence of a constant $C_1>0$, which depends only on $d,\D,$ and the coefficients of the matrices $A_{ii'}$,  such that
\begin{align}\label{eq:mycacc help}
|\widetilde u-\lambda|_{H^2(B_1(0))}\le C_1\big(\norm{\widetilde u-\lambda}{H^1(B_{3/2}(0))} +\norm{\widetilde f}{L^2(B_{3/2}(0))}\big).
\end{align}
Standard scaling arguments prove that
\begin{align*}
|\widetilde u-\lambda|_{H^2(B_1(0))}&\simeq \frac{r^2}{r^{d/2}}\, |u|_{H^2(B_r(x))}, \\
 \norm{\widetilde u-\lambda}{L^2(B_{3/2}(0))}&\simeq \frac{1}{r^{d/2}} \norm{u-\lambda}{L^2(B_{3r/2}(x))}, \\
\seminorm{\widetilde u-\lambda}{H^1(B_{3/2}(0))}&\simeq \frac{r}{r^{d/2}} \seminorm{u}{H^1(B_{3r/2}(x))},\\
\norm{\widetilde f}{L^2(B_{3/2}(0))}&\lesssim \frac{r^2}{r^{d/2}}\seminorm{u}{H^1(B_{3r/2}(x))}+\frac{r^2}{r^{d/2}}\norm{u-\lambda}{L^2(B_{3r/2}(x))}+r^2|\lambda|.
\end{align*}
Plugging this into \eqref{eq:mycacc help}, we obtain that
\begin{align}\label{eq:mycacc help2}
|u|_{H^2(B_r(x))}\lesssim\Big(\frac{1+r^2}{r^2}\norm{ u-\lambda}{L^2(B_{3r/2}(x))} +\frac{1+r}{r}| u|_{H^1(B_{3r/2}(x))}+r^{d/2}|\lambda|\Big).
\end{align}
We choose $\lambda$ as the integral mean $\lambda:=\int_{B_{3r/2}(x)} u(y) dy/|B_{3r/2}(x)|$.
The Cauchy--Schwarz inequality implies that 
\begin{align*}
|\lambda|\lesssim {\norm{u}{L^1(B_{3r/2}(x))}}/{|B_{3r/2}(x)|}\le \norm{u}{L^2(B_{3r/2}(x))}/{|B_{3r/2}(x)|}^{1/2}\simeq  r^{-d/2}\norm{u}{L^2(B_{3r/2}(x))}.
\end{align*}
Using this and the Poincar\'e inequality in \eqref{eq:mycacc help2}, we see that
\begin{align*}
|u|_{H^2(B_r(x))}\lesssim\Big(\norm{u}{L^2(B_{3r/2}(x))} +\frac{1+r+r^2}{r}| u|_{H^1(B_{3r/2}(x))}\Big).
\end{align*}
Together with  the fact that $B_{3r/2}(x)\subset B_{2r}(x)$, this concludes the proof.
\end{proof}


For the proof of the next proposition, we need the linear and continuous \textit{single-layer potential} from \cite[Theorem~6.11]{mclean}
\begin{align}\label{eq:single layer potential}
\widetilde{\mathfrak{V}}:H^{-1/2}(\Gamma)^D\to H^1(U)^D,
\end{align}
where $U$ is an arbitrary  bounded domain with $\Gamma\subset U$.
The single-layer operator $\mathfrak{V}$ is just the trace of $\mathfrak{\widetilde V}$, i.e., 
\begin{align}\label{eq:trace of V}
\mathfrak{V}=\widetilde{\mathfrak{V}}(\cdot)|_{\Gamma}:H^{-1/2}(\Gamma)^D\to H^{1/2}(\Gamma)^D;
\end{align}
see \cite[page 219--220]{mclean}.
Indeed, for $\psi\in L^\infty(\Gamma)$, \cite[page 201--202]{mclean} states the following integral representation
\begin{align}\label{eq:single layer potential integral}
(\widetilde{\mathfrak{V}}\psi)(x)=\int_{\Gamma} G(x-y) \psi(y) \,dy\quad\text{for all }x\in U.
\end{align}

\begin{proposition}\label{prop:invest for V}
Suppose \eqref{M:patch bem}--\eqref{M:semi bem}.
For $\TT_\bullet\in\T$, let $w_\bullet\in L^\infty(\Gamma)$ be a weight function which satisfies for some $\alpha>0$ and all $T\in\TT_\bullet$ that 
\begin{align}\label{eq:alpha-admissible}
\norm{w_\bullet}{L^\infty(T)}\le\alpha w_\bullet(x)\quad\text{for almost all }x\in\pi_\bullet(T).
\end{align}
Then, there exists a constant $\const{inv,\mathfrak{V}}>0$ such that for all $\psi\in L^2(\Gamma)^D$, it holds that  
\begin{align}\label{eq:invest general for V}
\norm{w_\bullet\nabla_\Gamma \mathfrak{V}\psi}{L^2(\Gamma)}\le \const{inv,\mathfrak{V}} \big(\norm{w_\bullet/h_\bullet^{1/2}}{L^\infty(\Gamma)}\norm{\psi}{H^{-1/2}(\Gamma)}+\norm{w_\bullet\psi}{L^2(\Gamma)}\big).
\end{align}
The constant $\const{inv,\mathfrak{V}}$ depends only on \eqref{M:patch bem}--\eqref{M:semi bem}, $\Gamma$,  the coefficients of $\mathfrak{P}$, and the admissibility constant $\alpha$.
The particular choice $w_\bullet=h_\bullet^{1/2}$ shows that
\begin{align}\label{eq:invest for V}
\norm{h_\bullet^{1/2}\nabla_\Gamma \mathfrak{V}\psi}{L^2(\Gamma)}\le \const{inv,\mathfrak{V}} \big(\norm{\psi}{H^{-1/2}(\Gamma)}+\norm{h_\bullet^{1/2}\psi}{L^2(\Gamma)}\big).
\end{align}
\end{proposition}

\begin{proof}
The proof works essentially as in \cite[Section 4]{invest}.
Therefore, we mainly  emphasize the differences and refer to \cite[Proposition~5.3.15]{diss} for further details.

By \eqref{M:cent bem} and with the abbreviation 
$\delta_1(T):={\diam(T)}/({2\const{cent}})$
and $U_T:=B_{\delta_1(T)}(T)$, there holds for all $T\in\TT_\bullet$ that $U_T\cap\Gamma\subset B_{2\delta_1(T)}(T)\cap\Gamma\subset\pi_\bullet(T)$.
This provides us with an open covering of $\Gamma\subset\bigcup_{T\in\TT_\bullet} U_T$. 
We show that this is even locally finite in the sense that there exists a constant $C>0$ with 
$\#\set{T\in\TT_\bullet}{x\in U_T}\le C$ for all  $x\in\R^d$:
Let $x\in\R^d$. Clearly, 
we may assume that $x\in\bigcup_{T\in\TT_\coarse}U_T$.
Choose $T_0\in\TT_\coarse$ with $x\in U_{T_0}$ such that $\delta_1(T_0)$ is minimal, and let $x_0\in T_0$ with $|x-x_0|<\delta_1(T_0)$.
If $T\in\TT_\coarse$   with $x\in U_T$, the triangle inequality yields that $\dist(\{x_0\},T)<2\delta_1(T)$.
By choice of $\delta_1(T)$, \eqref{M:cent bem} hence yields that $x_0\in \pi_\coarse(T)$.
Thus, $\set{T\in\TT_\coarse}{x\in U_T}\subseteq\set{T\in\TT_\coarse}{x_0\in\pi_\coarse(T)}$, and \eqref{M:patch bem}  implies that
\begin{align}\label{eq:local cover}
\#\set{T\in\TT_\coarse}{x\in U_T}\le\#\set{T\in\TT_\coarse}{x_0\in \pi_\coarse(T)}\le \const{patch}^2.
\end{align}

We fix (independently of $\TT_\bullet$) a bounded domain $U\subset\R^d$ with $U_T\subset U$ for all $T\in\TT_\bullet$.
We define for $T\in\TT_\bullet$ the near-field and the far-field of $u_{\mathfrak{V}}:=\widetilde{\mathfrak{V}}\psi$ by
\begin{align}\label{eq:u near and far}
u_{\mathfrak{V},T}^{\rm near}:=\widetilde{\mathfrak{V}}(\psi\chi_{\Gamma\cap U_T})\quad\text{and}\quad u_{\mathfrak{V},T}^{\rm far}:=\widetilde{\mathfrak{V}}(\psi \chi_{\Gamma\setminus U_T}).
\end{align}
In the following five steps, the near-field and the far-field are estimated separately.
The first two steps deal with the near-field, whereas the last three steps deal with the far-field.

{\bf Step 1:}
As in \cite[Lemma~4.1]{invest}, one shows that 
that for all $T\in\TT_\coarse$, all $\TT_\bullet$-piecewise (componentwise) constant functions $\Psi_\coarse^T\in\mathcal{P}^0(\TT_\bullet)^{\D}$ with $\supp(\Psi_\bullet^T)\subseteq \pi_\bullet(T)$ satisfy that 
\begin{align}\label{eq:C near tilde}
\norm{\widetilde{\mathfrak{V}} \Psi_\bullet^T}{H^1(U_T)}\lesssim  
\norm{h_\bullet^{1/2}\Psi_\bullet^T}{L^2(\pi_\bullet(T))}.
\end{align}
The  proof of \cite{invest} uses only \eqref{eq:local cover} and the fact that $|\nabla_x G(x-y)|\lesssim|x-y|^{-d+1}$ $(x,y)\in U\times\Gamma$ with $x\neq y$ for the Laplacian fundamental solution $G$.
However, according to \cite[Theorem 6.3 and Corollary 6.5]{mclean} this fact is also valid for general strongly elliptic second-order partial differential equations with $C^\infty$ coefficients.
Moreover, \cite{invest} bounds only the $H^1$-seminorm in \eqref{eq:C near tilde}, but the $L^2$-norm can be bounded similarly due to $|G(x-y)|\lesssim\max\{|\log|x-y||,|x-y|^{-d+2}\}\lesssim |x-y|^{-d+1}$ (see again \cite[Theorem~6.3 and Corollary~6.5]{mclean}).

\textbf{Step 2:}
With Step~1, one shows as in \cite[Proposition 4.2]{invest} 
that $u_{\mathfrak{V},T}^{\rm near}\in H^1(U)$ and $u_{\mathfrak{V},T}^{\rm near}|_{\Gamma}\in H^1(\Gamma)$
with 
\begin{align}\label{eq:C near}
\sum_{T\in\TT_\bullet}\norm{w_\bullet\nabla_\Gamma u_{\mathfrak{V},T}^{\rm near}}{L^2(T)}^2+\sum_{T\in\TT_\bullet}\norm{w_\bullet/h_\bullet^{1/2}}{L^\infty(T)}^2 \norm{ u_{\mathfrak{V},T}^{\rm near}}{H^1(U_T)}^2\lesssim
 \norm{w_\bullet \psi}{L^2(\Gamma)}^2.
\end{align}
In the proof, one applies the stability of ${\mathfrak{V}}:L^2(\Gamma)^{\D}\to H^1(\Gamma)^{\D}$ (see \eqref{eq:single layer operator})
 and $\widetilde {\mathfrak{V}}:H^{-1/2}(\Gamma)^{\D}\to H^1(U)^{\D}$ (see \eqref{eq:single layer potential}). 
Moreover, the approximation property \eqref{eq:cp41} is exploited by splitting $\psi\chi_{\Gamma\cap U_T}=P_\coarse(\psi\chi_{\Gamma\cap U_T}) + (1-P_\coarse)(\psi\chi_{\Gamma\cap U_T})$ and choosing $\Psi_\coarse^T:=P_\coarse(\psi\chi_{\gamma\cap U_T})$.
Note that \cite{invest} only proves \eqref{eq:C near} with $\seminorm{ u_{\mathfrak{V},T}^{\rm near}}{H^1(U_T)}^2$ instead of $\norm{ u_{\mathfrak{V},T}^{\rm near}}{H^1(U_T)}^2$.

{\bf Step 3:}
We consider the far-field.
We set $\Omega^{\rm ext}:=\R^d\setminus\overline\Omega$.
According to \cite[Theorem~6.11]{mclean}, for all $T\in\TT_\bullet$, the potential $u_{\mathfrak{V},T}^{\rm far}$ is a  solution of the transmission problem
\begin{align*}
\mathfrak{P} u_{\mathfrak{V},T}^{\rm far}&=0&&\text{on }\Omega\cup\Omega^{\rm ext},\\
[ u_{\mathfrak{V},T}^{\rm far}]_\Gamma&=0&&\text{in }H^{1/2}(\Gamma)^D,\\
[\mathfrak{D}_\nu u_{\mathfrak{V},T}^{\rm far}]_\Gamma&=-\psi\chi_{\Gamma\setminus U_T}&&\text{in }H^{-1/2}(\Gamma)^D,
\end{align*}
where $[\cdot]_\Gamma$ and \ $[\mathfrak{D}_\nu (\cdot)]_\Gamma$ denote the jump of the traces and the conormal derivatives respectively (see \cite[page~117]{mclean} for a precise definition)  across the boundary $\Gamma$.
Twofold integration by parts that uses these jump conditions shows that  $\mathfrak{P} u_{\mathfrak{V},T}^{\rm far}=0$ weakly on $U_T$.
Since $B_{\delta_1(T)}(x)\subseteq U_T$ for all $x\in T$,
 Lemma~\ref{lem:cacc} shows that $u_{\mathfrak{V},T}^{\rm far}\in C^\infty(B_{\delta_1(T)/2}(x))$  with 
\begin{align}\label{eq:cacc2}
|u_{\mathfrak{V},T}^{\rm far}|_{H^2(B_{\delta_1(T)/2}(x))}&\lesssim \norm{ u_{\mathfrak{V},T}^{\rm far}}{L^2(B_{\delta_1(T)}(x))}
+\diam(T)^{-1} | u_{\mathfrak{V},T}^{\rm far}|_{H^1(B_{\delta_1(T)}(x))}.
\end{align}
Note that \cite{invest} proves \eqref{eq:cacc2} even without the term $\norm{ u_{\mathfrak{V},T}^{\rm far}}{L^2(B_{\delta_1(T)}(x))}$. 
Indeed, since the kernel of the Laplace operator contains all constants, \cite{invest} employs a  Poincar\'e inequality to bound $\norm{ u_{\mathfrak{V},T}^{\rm far}}{L^2(B_{\delta_1(T)}(x))}$ by $\diam(T)^{-1} | u_{\mathfrak{V},T}^{\rm far}|_{H^1(B_{\delta_1(T)}(x))}$.

\textbf{Step 4:}
With inequality~\eqref{eq:cacc2} at hand, one can prove  the following local far-field bound for  the single-layer potential $\widetilde {\mathfrak{V}}$
\begin{align}\label{eq:C far tilde}
\norm{h_\bullet^{1/2}\nabla_\Gamma u_{\mathfrak{V},T}^{\rm far}}{L^2(T)}\le \norm{h_\bullet^{1/2} \nabla u_{\mathfrak{V},T}^{\rm far}}{L^2(T)}\lesssim
 \norm{u_{\mathfrak{V},T}^{\rm far}}{H^1(U_T)}.
\end{align}
The proof works as in \cite[Lemma 4.4]{invest}, and relies on a standard  trace inequality on $\Omega$, the   Caccioppoli inequality \eqref{eq:cacc2}  as well as the Besicovitch covering theorem.
Note that in \cite{invest}, the estimates \eqref{eq:cacc2} and thus \eqref{eq:C far tilde} even hold without the $L^2$-norm of $u_{\mathfrak{V},T}^{\rm far}$, since the Laplace problem is considered.
 
\textbf{Step 5:}
Finally,  \eqref{eq:local cover}, \eqref{eq:u near and far},   \eqref{eq:C near}, \eqref{eq:C far tilde}, and the stability of $\widetilde {\mathfrak{V}}: H^{-1/2}(\Gamma)^D\to H^1(U)^D$ (see \eqref{eq:single layer potential}) easily lead to the far-field bound for $\widetilde{\mathfrak{V}}$
\begin{align}\label{eq:C far}
\begin{split}
\sum_{T\in\TT_\bullet} \norm{w_\coarse \nabla_\Gamma  u_{\mathfrak{V},T}^{\rm far}}{L^2(T)}^2&\le \sum_{T\in\TT_\bullet}\norm{w_\coarse \nabla u_{\mathfrak{V},T}^{\rm far}}{L^2(T)}^2\\
&\lesssim \norm{w_\coarse/h_\bullet^{1/2}}{L^\infty(\Gamma)}^2 \norm{\psi}{H^{-1/2}(\Gamma)}^2+\norm{w_\coarse\psi}{L^2(\Gamma)}^2.
\end{split}
\end{align}
For the simple proof, we refer to \cite[Proposition~4.5]{invest}.
By definition~\eqref{eq:u near and far},  \eqref{eq:C far} together with \eqref{eq:C near} from Step 2 concludes the proof.
\end{proof}

\cite{invest} does not only treat  the single-layer operator $\mathfrak{V}:H^{-1/2}(\Gamma)^D\to H^{1/2}(\Gamma)^D$, but also derives similar inverse estimates as in \eqref{eq:invest general for V} for the double-layer operator, the adjoint double-layer operator, and the hyper-singular operator.
With similar techniques as in Proposition~\ref{prop:invest for V}, we will  also generalize this result in Appendix~\ref{sec:proof general invest}.
However, we will indeed only need the inverse estimate \eqref{eq:invest for V} for the single-layer operator in the remainder of the paper.

\subsection{Stability on non-refined elements~(\ref{item:stability})}\label{subsec:stability bem}
We show that the assumptions~\eqref{M:patch bem}--\eqref{M:semi bem} and~\eqref{S:inverse bem}--\eqref{S:nestedness bem} imply stability \eqref{item:stability}, i.e., the existence of $\const{stab}\ge1$ such that for all $\TT_\coarse\in\T$, and all $\TT_\fine\in\refine(\TT_\coarse)$,
it holds that 
\begin{align}\label{eq:e1 bem1}
|\eta_\fine(\TT_\coarse\cap\TT_\fine)-\eta_\coarse(\TT_\coarse\cap\TT_\fine)|\le \const{stab}\norm{\Phi_\fine-\Phi_\coarse}{H^{-1/2}(\Gamma)}.
\end{align}
In Section~\ref{subsec:reduction bem}, we will fix the constant $\const{\varrho}$ for $\varrho_{\coarse,\fine}$ defined in \eqref{item:stability} such that $\const{stab}\le\const{\varrho}$.
The reverse triangle inequality and the fact that $h_\fine=h_\coarse$ on $\omega:=\bigcup(\TT_\coarse\cap\TT_\fine)$ prove that 
\begin{align*}
&|\eta_\fine(\TT_\coarse\cap\TT_\fine)-\eta_\coarse(\TT_\coarse\cap\TT_\fine)|=\big|\norm{h_\fine^{1/2}\nabla_\Gamma\mathfrak{V}(\phi-\Phi_\fine)}{L^2(\omega)}-\norm{h_\coarse^{1/2}\nabla_\Gamma\mathfrak{V}(\phi-\Phi_\coarse)}{L^2(\omega)}\big|\\
&\quad\le \norm{h_\fine^{1/2}\nabla_\Gamma\mathfrak{V}(\Phi_\fine-\Phi_\bullet)}{L^2(\omega)}
\le \norm{h_\fine^{1/2}\nabla_\Gamma\mathfrak{V}(\Phi_\fine-\Phi_\bullet)}{L^2(\Gamma)}.
\end{align*}
\eqref{S:nestedness bem} shows that $\Phi_\fine-\Phi_\bullet\in\XX_\fine$. 
Therefore, the inverse inequalities from \eqref{S:inverse bem} and \eqref{eq:invest for V}  are applicable, which implies  \eqref{eq:e1 bem1}.
The constant $\const{stab}$ depends only on  $d,D$, $\Gamma$,  the coefficients of $\mathfrak{P}$, and the constants from \eqref{M:patch bem}--\eqref{M:semi bem} and  \eqref{S:inverse bem}.

\subsection{Reduction on refined elements~(\ref{item:reduction})}\label{subsec:reduction bem}
We show that the assumptions~\eqref{M:patch bem}--\eqref{M:semi bem}, \eqref{R:union bem}--\eqref{R:reduction bem}, and \eqref{S:inverse bem}--\eqref{S:nestedness bem} imply reduction on refined elements \eqref{item:reduction}, i.e., the existence of $\const{red}\ge1$ and $0<\ro{red}<1$ such that for all $\TT_\coarse\in\T$ and all $\TT_\fine\in\refine(\TT_\coarse)$, it holds that
\begin{align}\label{eq:e2 bem1}
\eta_\fine(\TT_\fine\setminus\TT_\coarse)^2
\le \ro{red} \, \eta_\coarse(\TT_\coarse\backslash\TT_\fine)^2
+ \const{red}\norm{\Phi_\fine-\Phi_\coarse}{H^{-1/2}(\Gamma)}^2.
\end{align}
With this, we can fix the constant  for $\varrho_{\coarse,\fine}$ defined in \eqref{item:stability} as  
\begin{align}\label{eq:crho defined2}
\const{\varrho}:=\max\{\const{stab},\const{red}^{1/2}\}.
\end{align}
Let $\omega:=\bigcup(\TT_\fine\setminus\TT_\coarse)$.
First, we apply the  triangle inequality
\begin{align*}
\eta_\fine(\TT_\fine\setminus\TT_\coarse)&= \norm{h_\fine^{1/2}\nabla_\Gamma\mathfrak{V}(\phi-\Phi_\fine)}{L^2(\omega)}\\
&\le  \norm{h_\fine^{1/2}\nabla_\Gamma\mathfrak{V}(\phi-\Phi_\coarse)}{L^2(\omega)}+  \norm{h_\fine^{1/2}\nabla_\Gamma\mathfrak{V}(\Phi_\fine-\Phi_\coarse)}{L^2(\omega)}.
\end{align*}
Clearly, \eqref{R:union bem}--\eqref{R:reduction bem} show that $\omega=\bigcup (\TT_\fine\setminus\TT_\coarse)=\bigcup( \TT_\coarse\setminus\TT_\fine)$ and $h_\fine\le \ro{son}^{1/(d-1)} h_\coarse$ on $\omega$.
Thus, we can proceed the estimate as follows
\begin{align*}
\eta_\fine(\TT_\fine\setminus\TT_\coarse)&\le \ro{son}^{1/(2d-2)} \norm{h_\coarse^{1/2}\nabla_\Gamma\mathfrak{V}(\phi-\Phi_\coarse)}{L^2(\omega)} + \norm{h_\fine^{1/2}\nabla_\Gamma\mathfrak{V}(\Phi_\fine-\Phi_\coarse)}{L^2(\omega)}\\
 &=\ro{son}^{1/(2d-2)}\eta_\coarse(\TT_\coarse\setminus\TT_\fine)+\norm{h_\fine^{1/2}\nabla_\Gamma\mathfrak{V}(\Phi_\fine-\Phi_\coarse)}{L^2(\omega)}.
\end{align*}
Since $\Phi_\coarse\in\XX_\coarse\subseteq\XX_\fine$ according to \eqref{S:nestedness bem}, we can apply the inverse estimates  \eqref{S:inverse bem} and \eqref{eq:invest for V}.
 Together with the Young inequality, we derive for arbitrary $\delta>0$ that 
\begin{align*}
\eta_\fine(\TT_\fine\setminus\TT_\coarse)^2\le(1+\delta)\ro{son}^{1/(d-1)}\eta_\coarse(\TT_\coarse\setminus\TT_\fine)^2+(1+\delta^{-1}) \const{inv,\mathfrak{V}}^2(1+ \const{inv})^2\norm{\Phi_\fine-\Phi_\coarse}{H^{-1/2}(\Gamma)}^2. 
\end{align*}
Choosing $\delta>0$ sufficiently small, we obtain \eqref{eq:e2 bem1}.
The constant $\const{red}$ depends only on  $d,D$, $\Gamma$, the coefficients of $\mathfrak{P}$, and the constants from \eqref{M:patch bem}--\eqref{M:semi bem}, \eqref{R:union bem}--\eqref{R:reduction bem}, and  \eqref{S:inverse bem}.

\subsection{General quasi-orthogonality (\ref{item:orthogonality})}\label{subsec:orthogonality bem}
According to \cite[Section~4.3]{axioms},   Section~\ref{subsec:perturbations bem}, Section~\ref{subsec:stability bem}, and Section~\ref{subsec:reduction bem} already imply estimator convergence $\lim_{\ell\to\infty}\eta_\ell=0$.
Therefore, reliability \eqref{eq:reliable bem} 
 implies error convergence $\lim_{\ell\to\infty}\norm{\phi-\Phi_\ell}{H^{-1/2}(\Gamma)}=0$.
In particular, we obtain that  $\phi\in\XX_\infty=\overline{\bigcup_{\ell\in\N_0}\XX_\ell}$.
Recall that we have already fixed the constant $\const{\varrho}$ in \eqref{eq:crho defined2}.
We  introduce the \textit{principal part} of $\mathfrak{P}$ as the corresponding  partial differential operator without lower-order terms
\begin{align}\label{eq:principal part}
\mathfrak{P}_0v:=
-\sum_{i=1}^d \sum_{i'=1}^d\partial_i (A_{ii'}\partial_{i'} v).
\end{align}
According to \cite[Lemma~4.5]{mclean}, the principal part is also coercive on $H_0^1(\Omega)^D$.
We denote its corresponding single-layer operator which can be defined as in Section~\ref{subsec:model problem bem} by
$\mathfrak{V}_0:H^{-1/2}(\Gamma)^D\to H^{1/2}(\Gamma)^D$.
Our assumption $A_{ii'}^\top=\overline {A_{i'i}}$ easily implies that $\mathfrak{V}_0$ is self-adjoint; see,  e.g., \cite[page~218]{mclean}.
With \eqref{eq:trace of V} and \eqref{eq:V difference} below, we particularly see the stability of $\mathfrak{V}-\mathfrak{V}_0:H^{-1/2}(\Gamma)^D\to H^{1-\varepsilon}(\Gamma)^D$ for all $\varepsilon>0$. 
Thus,  the Rellich compactness theorem \cite[Theorem~3.27]{mclean} implies  that $\mathfrak{V}-\mathfrak{V}_0:H^{-1/2}(\Gamma)^D\to H^{1/2}(\Gamma)^D$ is compact.
This yields that $\mathfrak{V}$ an elliptic operator which can be written as the sum of a self-adjoint operator $\mathfrak{V}_0$ plus a compact operator $\mathfrak{V}-\mathfrak{V}_0$.
From \cite{ffp14,helmholtz}, we thus derive the general quasi-orthogonality~\eqref{item:orthogonality} (see also \cite[Section~4.4.3]{diss} for all details),
i.e., the existence of
\begin{align}\label{eq:orth1 bem}
0\le\varepsilon_{\rm qo}<\sup_{\delta>0}\frac{1-(1+\delta)(1-(1-\ro{red})\theta)}{2+\delta^{-1}} 
\end{align}
and $\const{qo}\ge 1$ such that
\begin{align}\label{eq:orth2 bem}
\sum_{j=\ell}^{\ell+N}(\const{\varrho}\norm{\Phi_{j+1}-\Phi_j}{H^{-1/2}(\Gamma)}^2-\varepsilon_{\rm qo}\eta_j^2)\le\const{qo} \eta_\ell^2\quad\text{for all }\ell,N\in\N_0.
\end{align}

\begin{remark}\label{rem:E3 bem}
If the sesquilinear form 
 $\sprod{\mathfrak{V}\,\cdot}{\cdot}$ is Hermitian, \eqref{eq:orth2 bem}  follows from the Pythagoras theorem $\norm{\phi-\Phi_j}{\mathfrak{V}}^2+\norm{\Phi_{j+1}-\Phi_j}{\mathfrak{V}}^2=\norm{\phi-\Phi_j}{\mathfrak{V}}^2$  and norm equivalence
\begin{align*}
 \sum_{j=\ell}^{\ell+N}\norm{\Phi_{j+1}-\Phi_j}{H^{-1/2}(\Gamma)}^2
 \simeq \sum_{j=\ell}^{\ell+N}\norm{\Phi_{j+1}-\Phi_j}{\mathfrak{V}}^2
 = \norm{\phi-\Phi_\ell}{\mathfrak{V}}^2-\norm{\phi-\Phi_{\ell+N}}{\mathfrak{V}}^2 
 \lesssim \norm{\phi-\Phi_\ell}{H^{-1/2}(\Gamma)}^2.
\end{align*}
Together with reliability~\eqref{eq:reliable bem}, this proves~\eqref{eq:orth2 bem} even for $\varepsilon_{\rm qo}=0$, and $\const{qo}$ is independent of the sequence $(\Phi_\ell)_{\ell\in\N_0}$.
\end{remark}%

\subsection{Discrete reliability (\ref{item:discrete reliability})}
\label{subsec:discrete reliability bem}
The proof of \eqref{item:discrete reliability} is  inspired by  \cite[Proposition~5.3]{fkmp} which considers piecewise constants on  shape-regular triangulations as ansatz space.
Under the assumptions \eqref{M:patch bem}--\eqref{M:semi bem}, \eqref{eq:R:refine bem}, and \eqref{S:inverse bem}--\eqref{S:stab bem}, we show that there exist $\const{drel},\const{ref}\ge 1$ such that for all $\TT_\coarse\in\T$ and all $\TT_\fine\in\refine(\TT_\coarse)$, the subset 
\begin{align}\label{eq:RR defined}
\RR_{\coarse,\fine}:=\Pi_\coarse^{\q{supp}+\q{loc}+2}(\TT_\coarse\setminus\TT_\fine)
\end{align}
 satisfies  that
\begin{align*}
\const{\varrho}\norm{\Phi_\fine-\Phi_\coarse}{H^{-1/2}(\Gamma)}\le \const{drel}\,\eta_\coarse(\RR_{\coarse,\fine}),
\quad
\TT_\coarse\setminus\TT_\fine\subseteq\RR_{\coarse,\fine},
\text{ and }
\# \RR_{\coarse,\fine}\le\const{ref} (\#\TT_\fine-\#\TT_\coarse).
\end{align*}
The last two properties are obvious with $\const{ref}=\const{patch}^{\q{supp}+\q{loc}+2}$ by validity of~\eqref{M:patch bem} and \eqref{eq:R:refine bem}.
The first estimate is proved in three steps:

{\bf Step 1:}
For $\SS_1:=\TT_\coarse\cap\TT_\fine$, let $J_{\coarse,\SS_1}$ be the corresponding projection operator from \eqref{S:proj bem}--\eqref{S:stab bem}.
Ellipticity~\eqref{eq:ellipticity bem}, nestedness \eqref{S:nestedness bem} of the ansatz spaces,  and the definition \eqref{eq:pregalerkin bem} of the Galerkin approximations
yield that
\begin{align*}
\norm{\Phi_\fine-\Phi_\coarse}{H^{-1/2}(\Gamma)}^2&\lesssim \mathrm{Re}\, \sprod{\mathfrak{V}(\Phi_\fine-\Phi_\bullet)}{\Phi_\fine-\Phi_\coarse}_{L^2(\Gamma)}\\
&=\mathrm{Re} \,\sprod{\mathfrak{V}(\phi-\Phi_\bullet)}{(1-J_{\coarse,\SS_1})(\Phi_\fine-\Phi_\coarse)}_{L^2(\Gamma)}.
\end{align*}
\eqref{S:local bem} shows that $(\Phi_\fine-\Phi_\coarse)|_{\pi_\bullet^{\rm proj}(T)}\in \set{\Psi_\bullet|_{\pi_\bullet^{\rm proj}(T)}}{\Psi_\bullet\in\XX_\bullet}$ for any $T\in\TT_\coarse\setminus\Pi_\coarse^{\q{loc}}(\TT_\coarse\setminus\TT_\fine)$.
Moreover, one easily sees  that 
\begin{align}\label{eq:drel bem12}
\Pi_\coarse^{\q{loc}}(T)\subseteq\TT_\coarse\cap\TT_\fine=\SS_1\quad \text{for all }T\in\TT_\coarse\setminus\Pi_\coarse^{\q{loc}}(\TT_\coarse\setminus\TT_\fine).
\end{align}
Hence, the local projection property \eqref{S:proj bem} of $J_{\coarse,\SS_1}$ is applicable and proves that
$J_{\coarse,\SS_1}(\Phi_\fine-\Phi_\coarse)=\Phi_\fine-\Phi_\coarse$ on $\Gamma\setminus\pi_\coarse^{\q{loc}}(\TT_\bullet\setminus\TT_\fine)$.
With ${\SS_2}:=\Pi_\coarse^{\q{loc}}(\TT_\bullet\setminus\TT_\fine)$, Lemma~\ref{lem:tildechi} below provides a smooth cut-off function $\chi:=\widetilde\chi_{\mathcal{S}_2}\in H^1(\Gamma)$
with $0\le \chi\le1$ such that $\chi=1$ on $\bigcup\mathcal{S}_2$, $\chi=0$ on $\Gamma\setminus\pi_\coarse(\SS_2)$, and $|\nabla_\Gamma\chi|\lesssim h_\coarse^{-1}$, where the hidden constant depends only on $d$ and \eqref{M:patch bem}--\eqref{M:cent bem}.
This leads to
%
\begin{align}\label{eq:drel help}
 \norm{\Phi_\fine-\Phi_\coarse}{H^{-1/2}(\Gamma)}^2&\lesssim\mathrm{Re} \, \sprod{\chi\,\mathfrak{V}(\phi-\Phi_\bullet)}{(1-J_{\coarse,\SS_1})(\Phi_\fine-\Phi_\coarse)}_{L^2(\Gamma)}.
\end{align}
We bound the  two terms ${\rm I}:=\mathrm{Re}\,\sprod{\chi\,\mathfrak{V}(\phi-\Phi_\bullet)}{\Phi_\fine-\Phi_\coarse}_{L^2(\Gamma)}$ and ${\rm II}:=\mathrm{Re}\,\sprod{\chi\,\mathfrak{V}(\phi-\Phi_\bullet)}{J_{\coarse,\SS_1}(\Phi_\fine-\Phi_\coarse)}_{L^2(\Gamma)}$ separately.
Since $H^{-1/2}(\Gamma)^D$ is the dual space of $H^{1/2}(\Gamma)^D$, it holds that 
\begin{align}\label{eq:drel help1}
 {\rm I}
 \le \norm{\chi\,\mathfrak{V}(\phi-\Phi_\bullet)}{H^{1/2}(\Gamma)}\norm{\Phi_\fine-\Phi_\coarse}{H^{-1/2}(\Gamma)}.
\end{align}
The Cauchy--Schwarz inequality shows that
\begin{align*}
{\rm II}\le\norm{h_\coarse^{-1/2}\chi\,\mathfrak{V}(\phi-\Phi_\bullet)}{L^2(\Gamma)}\norm{h_\coarse^{1/2}J_{\coarse,\SS_1}(\Phi_\fine-\Phi_\coarse)}{L^2(\Gamma)}.
\end{align*}
Since $J_{\coarse,\SS_1}:L^2(\Gamma)^D\to\set{\Psi_\coarse\in\XX_\coarse}{\Psi_\coarse|_{\bigcup(\TT_\coarse\setminus{\SS_1})}=0}$, it holds that $\supp(J_{\coarse,\SS_1}(\Phi_\fine-\Phi_\coarse))\subseteq \bigcup(\TT_\coarse\cap\TT_\fine)$. 
This together with  the fact that $h_\bullet=h_\fine$ on $\bigcup(\TT_\bullet\cap\TT_\fine)$,  the local $L^2$-stability \eqref{S:stab bem} and \eqref{M:patch bem}--\eqref{M:shape bem} implies that
\begin{align*}
{\rm II}&=\norm{h_\coarse^{-1/2}\chi\,\mathfrak{V}(\phi-\Phi_\bullet)}{L^2(\Gamma)}\norm{h_\fine^{1/2}J_{\coarse,\SS_1}(\Phi_\fine-\Phi_\coarse)}{L^2(\bigcup(\TT_\coarse\cap\TT_\fine))}\\
&\lesssim\norm{h_\coarse^{-1/2}\chi\,\mathfrak{V}(\phi-\Phi_\bullet)}{L^2(\Gamma)}\norm{h_\fine^{1/2}(\Phi_\fine-\Phi_\coarse)}{L^2(\Gamma)}.
\end{align*}
With the inverse inequality \eqref{S:inverse bem} applied to $\Phi_\fine-\Phi_\coarse\in\XX_\fine$ (see \eqref{S:nestedness bem}),  the latter estimate implies that
\begin{align}\label{eq:drel help2}
{\rm II}
\lesssim \norm{h_\coarse^{-1/2} \chi\mathfrak{V}(\phi-\Phi_\coarse)}{L^2(\Gamma)}\norm{\Phi_\fine-\Phi_\coarse}{H^{-1/2}(\Gamma)}.
\end{align}
Plugging \eqref{eq:drel help1}--\eqref{eq:drel help2} into \eqref{eq:drel help} shows that 
\begin{align}\label{eq:drel help3}
\begin{split}
\norm{\Phi_\fine-\Phi_\coarse}{H^{-1/2}(\Gamma)}&\lesssim \norm{h_\coarse^{-1/2}\chi\,\mathfrak{V}(\phi-\Phi_\bullet)}{L^2(\Gamma)}+\norm{\chi\,\mathfrak{V}(\phi-\Phi_\bullet)}{H^{1/2}(\Gamma)}.
\end{split}
\end{align}

{\bf Step 2:}
Next, we deal with the first summand of \eqref{eq:drel help3}.
With $\supp(\chi)\subseteq \pi_\coarse^{\q{loc}+1}(\TT_\coarse\setminus\TT_\fine) $ and $0\le\chi\le 1$, 
this implies that
\begin{align}\label{eq:chi away}
 \norm{h_\coarse^{-1/2}\chi\,\mathfrak{V}(\phi-\Phi_\bullet)}{L^2(\Gamma)}\le\norm{h_\coarse^{-1/2}\mathfrak{V}(\phi-\Phi_\bullet)}{L^2\big(\pi_\coarse^{\q{loc}+1}(\TT_\coarse\setminus\TT_\fine)\big)}.
\end{align}
By Galerkin orthogonality \eqref{eq:galerkin bem}, we see that $\mathfrak{V}(\phi-\Phi_\coarse)$ is $L^2$-orthogonal to all functions of $\XX_\bullet$ which includes in particular the functions $\Psi_{\coarse,T,j}$ from \eqref{S:unity bem}.
Hence, we can apply Proposition~\ref{prop:poincare}.
Together with \eqref{M:patch bem}--\eqref{M:shape bem} and recalling \eqref{eq:RR defined}, \eqref{eq:chi away} proves that
\begin{align*}
 \norm{h_\coarse^{-1/2}\chi\,\mathfrak{V}(\phi-\Phi_\bullet)}{L^2(\Gamma)}
 \lesssim \norm{h_\coarse^{1/2}\nabla_\Gamma \mathfrak{V}(\phi-\Phi_\coarse)}{L^2\big(\pi^{\q{supp}+\q{loc}+2}_\bullet(\TT_\coarse\setminus\TT_\fine)\big)}
 =\eta_\coarse(\RR_{\coarse,\fine}).
\end{align*}

{\bf Step 3:}
It remains to consider the second summand of   \eqref{eq:drel help3}.
Lemma \ref{lem:first split}  in conjunction with shape-regularity  \eqref{M:shape bem} implies that 
\begin{align*}
\norm{\chi\,\mathfrak{V}(\phi-\Phi_\bullet)}{H^{1/2}(\Gamma)}^2\lesssim
\sum_{T\in\TT_\bullet}\sum_{T'\in\Pi_\bullet(T)}|\chi\,\mathfrak{V}(\phi-\Phi_\bullet)|_{H^{1/2}(T\cup T')}^2 
+ \norm{h_\coarse^{-1/2}\chi\,\mathfrak{V}(\phi-\Phi_\bullet)}{L^2(\Gamma)}.
\end{align*}
We have already dealt with the second summand in Step 2 (see~\eqref{eq:chi away}).
For the first one, we fix again some $z(T,T')\in T\cap T'$ for any $T\in\TT_\coarse,T'\in\Pi_\coarse(T)$.
\eqref{M:patch bem}--\eqref{M:shape bem} and \eqref{M:semi bem} show that
\begin{align*}
&\sum_{T\in\TT_\bullet}\sum_{T'\in\Pi_\bullet(T)}|\chi\,\mathfrak{V}(\phi-\Phi_\bullet)|_{H^{1/2}(T\cup T')}^2\le\sum_{T\in\TT_\bullet}\sum_{T'\in\Pi_\bullet(T)}|\chi\,\mathfrak{V}(\phi-\Phi_\bullet)|_{H^{1/2}(\pi_\coarse(z(T,T')))}^2\\
&\quad\le\sum_{T\in\TT_\bullet}\sum_{T'\in\Pi_\bullet(T)}\norm{h_\bullet^{1/2}\nabla_\Gamma(\chi\,\mathfrak{V}(\phi-\Phi_\bullet))}{L^2(\pi_\bullet(z(T,T'))}^2\lesssim \norm{h_\bullet^{1/2}\nabla_\Gamma(\chi\,\mathfrak{V}(\phi-\Phi_\bullet))}{L^2(\Gamma)}^2.
\end{align*}
With the product rule and \eqref{eq:nablachi}, we continue our estimate
\begin{align*}
\norm{\chi\,\mathfrak{V}(\phi-\Phi_\bullet)}{H^{1/2}(\Gamma)}^2\lesssim  \norm{h_\bullet^{-1/2}\mathfrak{V}(\phi-\Phi_\bullet))}{L^2(\supp(\chi))}^2+ \norm{h_\bullet^{1/2}\nabla_\Gamma\mathfrak{V}(\phi-\Phi_\bullet))}{L^2(\supp(\chi))}^2.
\end{align*}
Note that we have already dealt with the first summand in Step 2 (see \eqref{eq:chi away}).
Finally, $\supp(\chi)\subseteq \pi_\coarse^{\q{loc}+1}(\TT_\coarse\setminus\TT_\fine)$ and $\Pi_\coarse^{\q{loc}+1}(\TT_\coarse\setminus\TT_\fine)\subseteq\RR_{\coarse,\fine}$ (see \eqref{eq:RR defined}) prove for the second one that
\begin{align*}
\norm{h_\bullet^{1/2}\nabla_\Gamma\mathfrak{V}(\phi-\Phi_\bullet))}{L^2(\supp(\chi))}^2\le \eta_\bullet(\RR_{\coarse,\fine})^2.
\end{align*}
With this, we conclude the proof of discrete reliability {\rm \eqref{item:discrete reliability}}.
The constant $\const{drel}$ depends only on  $\const{\varrho}, d,D$, $\Gamma$, and  the constants from \eqref{M:patch bem}--\eqref{M:semi bem} and  \eqref{S:inverse bem}--\eqref{S:stab bem}.



\appendix

\section{Smooth characteristic functions}\label{eq:proof of general hat}

The following lemma provides a smooth cut-off function as used in the proof of discrete reliability~\eqref{item:discrete reliability}; see Section~\ref{subsec:discrete reliability bem}.
For regular simplicial meshes as in Section~\ref{sec:standard applications}, such functions can easily be constructed by means of hat functions; see \cite[Section~5.3]{invest}.
For the present abstract setting, the construction is more technical.

\begin{lemma}\label{lem:tildechi}
Let $\TT_\coarse\in\T$ and $\SS\subseteq\TT_\coarse$.
Suppose \eqref{M:patch bem}--\eqref{M:cent bem}.
Then there exists a  function 
 $\widetilde\chi_{\SS}\in H^1(\Gamma)$  such that for  almost all $x\in\Gamma$
\begin{subequations}
\begin{align}\label{eq:tildechi_a}
&\widetilde\chi_{\SS}(x)=1\quad&&\text{if }x\in\bigcup\SS,\\
\label{eq:tildechi_b}
&0\le \widetilde\chi_{\SS}(x)\le 1\quad&&\text{if }x\in\pi_\coarse(\SS),\\
\label{eq:tildechi_c}
&\widetilde\chi_{\SS}(x)=0\quad&&\text{if }x\not \in \pi_\coarse(\SS).
\end{align}
\end{subequations}
Further, there exists a constant $C>0$ such that 
\begin{align}\label{eq:nablachi}
|\nabla_\Gamma \widetilde\chi_{\SS}(x)|\le C h_\coarse(x)^{-1}
\quad
\text{for almost all } x\in\Gamma.
\end{align}
The constant $C$ depends only on the dimension $d$ and the constants from \eqref{M:patch bem}--\eqref{M:cent bem}.
\end{lemma}

\begin{proof}
In the following  three steps, we will even prove the existence of a function 
$\widetilde\chi_\SS\in C^\infty(O)$ with an open superset $O\supset\Gamma$ such that the restriction to $\Gamma$ has the desired properties.
With the constants from  \eqref{M:patch bem}--\eqref{M:locuni bem} and \eqref{M:cent bem}, we define for all $T\in\TT_\bullet$,
\begin{align}\label{eq:deltas}
\delta_1(T):=\frac{\diam(T)}{2\const{patch} \const{locuni}\const{cent}},\quad\delta_2(T)&:=\frac{\diam(T)}{\const{cent}}, \quad\delta_3(T):=\frac{\diam(T)}{2\const{cent}}.
\end{align}

{\bf Step 1:} 
First, we construct an equivalent smooth mesh-size function $\delta_\coarse\in C^\infty(\R^d)$ with uniformly bounded gradient on $\Gamma$.
Let $K_1\in C^\infty(\R^d)$ be a standard mollifier with $0\le K_1\le 1$ on $B_1(0)$, $K_1=0$ on $\R^d\setminus B_1(0)$, and $\int_{\R^d} K_1\,dx=1$.
For $s>0$, we set $K_s(\cdot):=K_1(\cdot/s) s^{-d}$.
By convolution, we define 
\begin{align}
\delta_\coarse:=\sum_{T\in\TT_\bullet} \delta_1(T)\, \chi _{B_{\delta_2(T)}(T)}*K_{\delta_2(T)}.
\end{align}
Note that $\supp(\chi _{B_{\delta_2(T)}(T)}*K_{\delta_2(T)})\subseteq B_{2\delta_2(T)}(T)$ for $T\in\TT_\bullet$.
Thus,  \eqref{M:cent bem} and the choice \eqref{eq:deltas} of $\delta_2(T)$ yields that $\supp(\chi _{B_{\delta_2(T)}(T)}*K_{\delta_2(T)})\cap\Gamma\subseteq\pi_\bullet(T)$.
Together with \eqref{M:patch bem}--\eqref{M:locuni bem} and $0\le \chi _{B_{\delta_2(T)}(T)}*K_{\delta_2(T)}\le 1$, this implies  for  the interior ${\rm int}(T')$ of any  $T'\in\TT_\bullet$  that
\begin{align*}
\delta_\bullet|_{{\rm int}(T')}\le \sum_{T\in\TT_\bullet} \delta_1(T)\chi_{\pi_\bullet(T)}|_{{\rm int}(T')}=\sum_{T\in\Pi_\bullet(T')} \delta_1(T)\le \const{patch} \const{locuni}\,\delta_1(T').
\end{align*}
Note that the restriction to the interior is indeed necessary since the second term is discontinuous across faces of $\TT_\coarse$.
However, by continuity of $\delta_\coarse$, this estimate is also satisfied if ${\rm int}(T')$ is replaced by $T'$, i.e., $\delta_\coarse|_{T'}\le \const{patch}\const{locuni}\delta_1(T')$.
The fact that $\chi _{B_{\delta_2(T')}(T')}*K_{\delta_2(T')}=1$ on $T'$ shows that also the converse estimate is valid.
This leads to
\begin{align}\label{eq:tildedelta estimate}
\frac{\diam(T')}{2\const{patch} \const{locuni}\const{cent}}
=\delta_1(T')\le \delta_{\bullet}|_{T'}\le  \const{patch} \const{locuni}\,\delta_1(T')=\delta_3(T')
\quad\text{for all }T'\in\TT_\bullet.
\end{align}
In particular, this proves the existence of an  open set $\R^d\supset O\supset\Gamma$ such that $\delta_{\bullet}>0$ on $O$.
Finally, we consider the gradient of $\delta_\coarse$ for $x\in\Gamma$.
Recall that $\supp(\chi _{B_{\delta_2(T)}(T)}*K_{\delta_2(T)})\subseteq\pi_\bullet(T)$.
Together with the H\"older inequality, $\norm{\nabla K_s}{L^1(\R^d)}\lesssim s^{-1}$, and  \eqref{M:patch bem}--\eqref{M:locuni bem}, this proves that
\begin{align}\label{eq:nablatildedelta estimate}
\begin{split} 
|\nabla\delta_\coarse(x)|
&=\sum_{T\in\TT_\bullet} \delta_1(T)\,  \chi_{\pi_\bullet(T)}(x)\,|\chi _{B_{\delta_2(T)}(T)}*\nabla K_{\delta_2(T)}(x)|\\
&\lesssim\sum_{T\in\TT_\bullet} \delta_1(T)\, \chi_{\pi_\bullet(T)}(x)\delta_2(T)^{-1}\lesssim 1.
\end{split}
\end{align}

{\bf Step 2:}
In this step, we construct $\widetilde\chi_{\SS}$ and prove~\eqref{eq:tildechi_a}--\eqref{eq:tildechi_c}.
For $x\in O$, we define the quasi-convolution 
\begin{align*}
\widetilde \chi_{\SS}(x):= \int_{\R^d} \chi_{\widetilde S}(y) \,K_{\delta_\coarse(x)}(x-y)\,dy, \quad\text{where }\widetilde S:=\bigcup \set{B_{\delta_3(T)}(T)}{T\in\SS}
\end{align*}
Since $\delta_\coarse>0$ on $O$, the chain rule shows that any derivative with respect to $x$ of the term $K_{\delta_\coarse(x)}(x-y)$ exists and is continuous at all $(x,y)\in O\times \widetilde S$.
This yields that $\widetilde \chi_{\SS}\in C^\infty(O)$ and  $\widetilde\chi_\SS|_\Gamma\in H^1(\Gamma)$; see, e.g., \cite[page~98--99]{mclean}.
Since $\supp(K_s)=B_s(0)$, it holds that
\begin{align}\label{eq:tildechi2}
\widetilde \chi_{\SS}(x)= \int_{B_{\delta_\coarse(x)}(x)} \chi_{\widetilde S}(y) \,K_{\delta_\coarse(x)}(x-y)\,dy.
\end{align}
If $B_{\delta_\coarse(x)}(x)\subseteq\widetilde S$ and thus $ \chi_{\widetilde S}(y)=1$, the properties of $K_1$ show that  $\widetilde\chi_\SS(x)=1$.
Due to $\delta_\coarse|_{T'}\le\delta_{3}(T')$ for all $T'\in\TT_\bullet$ (which follows from \eqref{eq:tildedelta estimate}), this is particularly satisfied if $x\in\bigcup\SS$.
This proves \eqref{eq:tildechi_a}.
Moreover, \eqref{eq:tildechi2} shows that $0\le\widetilde\chi_\SS(x)\le 1$ for all $x\in\R^d$ (and hence verifies \eqref{eq:tildechi_b}), and $\widetilde\chi_\SS(x)=0$ if $B_{\delta_\coarse(x)}(x)\cap\widetilde S=\emptyset$.
For~\eqref{eq:tildechi_c}, it thus remains to prove that $x\in\Gamma\setminus\pi_\coarse(\SS)$ implies that  $B_{\delta_\coarse(x)}(x)\cap\widetilde S=\emptyset$.
We prove the contraposition.
Let $x\in\Gamma$ and suppose that $B_{\delta_\coarse(x)}(x)\cap\widetilde S\neq\emptyset$. Then, there exists $T\in\SS$ and $y\in \R^d$ such that $|x-y|<\delta_\coarse(x)$ and $\dist(\{y\},T)<\delta_3(T)$.
The triangle inequality yields that 
\begin{align}\label{eq:max delta}
\dist(\{x\},T)\le |x-y| +\dist(\{y\},T)<\delta_\coarse(x)+\delta_3(T)\le 2\max\big\{\delta_\coarse(x),\delta_3(T)\big\}.
\end{align}
Now, we differ two different cases:
If $\delta_\coarse(x)\le \delta_3(T)$, then we have that $\dist(\{x\},T)< 2\delta_3(T)$.
The choice \eqref{eq:deltas} of $\delta_3(T)$ together with \eqref{M:cent bem} shows that $x\in\pi_\bullet(T)\subseteq\pi_\bullet(\SS)$.
If $\delta_\coarse(x)> \delta_3(T)$, then we have that $\dist(\{x\},T)<2\delta_\coarse(x)$.
Let $T'\in\TT_\bullet$ with $x\in T'$ and $z\in T$ with $|x-z|=\dist(\{x\},T)$.
Together with \eqref{eq:tildedelta estimate} and  \eqref{eq:max delta}, this yields that
\begin{align*}
\dist(\{z\},T')\le |x-z|=\dist(\{x\},T)< 2\delta_\coarse(x)\le 2\const{patch} \const{locuni}\,\delta_{1}(T')
=\frac{\diam(T')}{\const{cent}}.
\end{align*}
Hence, \eqref{M:cent bem} implies that $z\not\in\Gamma\setminus\pi_\coarse(T')$ and thus $z\in\pi_\coarse(T')$.
Any $z'\in\Gamma$ that is sufficiently close to $z$ also satisfies that $\dist(\{z'\},T')<\diam(T')/\const{cent}$ and thus $z'\in \pi_\coarse(T')$.
Since $z\in T=\overline{{\rm int}(T)}$, $z'$ can be particularly chosen in ${\rm int}(T)$.
Hence, we see that ${\rm int}(T)\cap\pi_\coarse(T')\neq\emptyset.$
This is equivalent to $T'\in\Pi_\coarse(T)$,  which concludes that $x\in T'\subseteq\pi_\bullet(T)\subseteq\pi_\bullet(\SS)$.

{\bf Step 3:}
Finally, we prove \eqref{eq:nablachi}.
We recall that $\delta_\coarse>0$ on $O$; see Step 1.
With the identity matrix $I\in\R^{d\times d}$ and the matrix $(x-y)(\nabla \delta_\coarse(x))^\top\in \R^{d\times d}$, elementary calculations prove for all $x\in O\supset\Gamma$  and all $y\in\R^d$ that
\begin{align*}%
\big(\nabla_x \big[K_{\delta_\coarse(x)}(x-y)\big]\big)^\top&=\Big[\nabla K_1\Big(\frac{x-y}{\delta_\coarse(x)}\Big)\Big]^\top \,\frac{\delta_\coarse(x)I-(x-y)(\nabla \delta_\coarse(x))^\top}{\delta_\coarse(x)^2} \,\delta_\coarse(x)^{-d}\\
&\quad+K_1\Big(\frac{x-y}{\delta_\coarse(x)}\Big) \delta_\coarse(x)^{-d-1}(-d) (\nabla \delta_\coarse(x))^\top.
\end{align*}%
Considering the norm, we see that
\begin{align*}
\big|\nabla_x \big(K_{\delta_\coarse(x)}(x-y)\big)\big|\lesssim\delta_\coarse(x)^{-d-1}+ |x-y| |\nabla \delta_\coarse(x)|\,\delta_\coarse(x)^{-d-2}+ \delta_\coarse(x)^{-d-1} |\nabla \delta_\coarse(x)|.
\end{align*}
Together with $\supp(K_s)=B_s(0)$, this shows for all $x\in\Gamma$ that
\begin{align*}
|\nabla \widetilde\chi_\SS(x)|&=\Big|\int_{\R^d} \chi_{\widetilde S}(y) \nabla_x \big(K_{\delta_\coarse(x)}(x-y)\big)\,dy\Big|
\\&\lesssim\int_{B_{\delta_\coarse(x)}(x)} \delta_\coarse(x)^{-d-1}+ |x-y| |\nabla \delta_\coarse(x)|\,\delta_\coarse(x)^{-d-2}+ \delta_\coarse(x)^{-d-1} |\nabla \delta_\coarse(x)| \,dy\\
&\lesssim  \delta_\coarse(x)^{-1}(1+\norm{\nabla\delta_\coarse}{L^\infty(\Gamma)}).
\end{align*}
Thus, \eqref{eq:tildedelta estimate}--\eqref{eq:nablatildedelta estimate} and \eqref{M:shape bem} prove that
$|\nabla \widetilde\chi_\SS(x)|\lesssim h_\bullet(x)^{-1}$ for almost all $x\in\Gamma$.
Moreover, for smooth functions, the surface gradient $\nabla_\Gamma$ is the orthogonal projection of the gradient $\nabla$ onto the tangent plane; see, e.g., \cite[Lemma 2.22]{gme}).
With the outer normal vector $\nu$, this implies that $\nabla_\Gamma\widetilde\chi_\SS =\nabla \widetilde\chi_\SS-(\nabla\widetilde\chi_\SS\cdot\nu)\nu$ almost everywhere on $\Gamma$, and concludes the proof with the previous estimate.
\end{proof}

\section{Inverse inequalities for other integral operators}\label{sec:proof general invest}
In Proposition~\ref{prop:invest for V}, we have generalized an inverse estimate from \cite{invest} for  the single-layer operator $\mathfrak{V}:H^{-1/2}(\Gamma)^D\to H^{1/2}(\Gamma)^D$.
\cite{invest} additionally derived similar estimates for the \emph{double-layer operator} $\mathfrak{K}':H^{1/2}(\Gamma)^D\to H^{1/2}(\Gamma)^D$, the \emph{adjoint double-layer operator} $\mathfrak{K}':H^{-1/2}(\Gamma)^D\to H^{-1/2}(\Gamma)^D$, and the \emph{hyper-singular operator} $\mathfrak{W}:H^{1/2}(\Gamma)^D\to H^{-1/2}(\Gamma)^D$; see, e.g., \cite[page~218]{mclean}  for a precise definition (where these operators are denoted by $\frac{1}{2}T,\frac{1}{2}\widetilde T^*,$ and $R$, respectively).
Although \cite{invest}  considered only the Laplace problem,  the techniques of the proof of Proposition~\ref{prop:invest for V} extend the result at least to partial differential operators  without lowest-order term $c u$. 
With some further effort, one can even prove it for arbitrary PDE operators $\mathfrak{P}$ with constant coefficients as in Section~\ref{subsec:model problem bem}.
To this end, one requires additional regularity of the trace operator $(\cdot)|_\Gamma: H^{3/2}(\Omega)^D\to H^1(\Gamma)^D$, which is satisfied for piecewise smooth boundaries $\Gamma$; see, e.g., \cite[Remark~3.1.18]{ss}.

For the proof, we will frequently use \cite[Theorem~4.24]{mclean}, which reads as follows: 
Let $u\in H^1(\Omega)$ be arbitrary with $\mathfrak{P} u\in L^2(\Omega)^D$ in the weak sense and $u|_\Gamma\in H^1(\Gamma)$.
Then, $\mathfrak{D}_\nu u\in L^2(\Gamma)^D$ and 
\begin{align}\label{eq:regular normal}
\norm{\mathfrak{D}_\nu u}{L^2(\Gamma)}
\lesssim \norm{u|_\Gamma}{H^1(\Gamma)} + \norm{u}{H^1(\Omega)} + \norm{\mathfrak{P}u}{L^2(\Omega)}.
\end{align}

\begin{proposition}\label{prop:general invest}
Suppose \eqref{M:patch bem}--\eqref{M:semi bem}.
For $\TT_\bullet\in\T$, let $w_\bullet\in L^\infty(\Gamma)$ be a weight function which satisfies for some $\alpha>0$ and all $T\in\TT_\bullet$ that 
\begin{align}
\norm{w_\bullet}{L^\infty(T)}\le\alpha w_\bullet(x)\quad\text{for almost all }x\in\pi_\bullet(T).
\end{align}
Then, there exists a constant $C_1>0$ such that for all $\psi\in L^2(\Gamma)^D$ and $v\in H^1(\Gamma)^D$, 
\begin{align}
\norm{w_\bullet\nabla_\Gamma \mathfrak{V}\psi}{L^2(\Gamma)}+\norm{w_\bullet \mathfrak{K}'\psi}{L^2(\Gamma)}&\le C_1 \big(\norm{w_\bullet/h_\bullet^{1/2}}{L^\infty(\Gamma)}\norm{\psi}{H^{-1/2}(\Gamma)}+\norm{w_\bullet\psi}{L^2(\Gamma)}\big),\label{eq:invest1}
\end{align}
If we additionally suppose that the  trace operator satisfies the stability $(\cdot)|_\Gamma:H^{3/2}(\Omega)^D\to H^1(\Gamma)^D$, there exists a constant $C_2>0$ such that for all $\psi\in L^2(\Gamma)^D$ and $v\in H^1(\Gamma)^D$,
\begin{align}
\norm{w_\bullet\nabla_\Gamma \mathfrak{K}v}{L^2(\Gamma)}+\norm{w_\bullet \mathfrak{W}v}{L^2(\Gamma)}&\le C_2 \big(\norm{w_\bullet/h_\bullet^{1/2}}{L^\infty(\Gamma)}\norm{v}{H^{1/2}(\Gamma)}+\norm{w_\bullet\nabla_\Gamma v}{L^2(\Gamma)}\big).\label{eq:invest2}
\end{align}
The constants $C_1$ and $C_2$ depend only on \eqref{M:patch bem}--\eqref{M:semi bem}, $\Gamma$,  the coefficients of $\mathfrak{P}$, and the admissibility constant $\alpha$.
\end{proposition}

\begin{proof}[Proof of \eqref{eq:invest1}]
The bound for $\norm{w_\coarse\nabla_\Gamma\mathfrak{V}\psi}{L^2(\Gamma)}$ is just the assertion of Proposition~\ref{prop:invest for V}.
With the results from the proof of Proposition~\ref{prop:invest for V}, one can easily estimate the second summand $\norm{w_\bullet \mathfrak{K}'\psi}{L^2(\Gamma)}$ as in \cite[Section~6]{invest}.
In particular, the stability of $\mathfrak{K}':L^2(\Gamma)\to L^2(\Gamma)$ is exploited, which is stated in \cite[page~209]{mclean}. 
Further, one uses the fact from \cite[page~218]{mclean} that $\mathfrak{K}'=\widetilde{\mathfrak{D}}^{\rm int}_\nu \widetilde{\mathfrak{V}}-1/2$, where $\widetilde{\mathfrak{D}}_\nu^{\rm int}(\cdot)$ denotes the interior modified conormal derivative from \cite[page~117--118]{mclean}. 
\end{proof}
\begin{proof}[Proof of \eqref{eq:invest2}]
As in the  proof of Proposition~\ref{prop:invest for V}, we abbreviate $\delta_1(T):={\diam(T)}/({2\const{cent}})$
and $U_T:=B_{\delta_1(T)}(T)$ for all $T\in\TT_\coarse$.
Let $v_T:=|T|^{-1}\int_T v\,dx$.
Further, let $\widetilde\chi_T:=\widetilde \chi_{\{T\}}$ be the smooth quasi-indicator function of $T$ from Lemma~\ref{lem:tildechi}.
With the Poincar\'e inequality~\eqref{eq:SS poincare}, the localization properties \eqref{eq:first split} as well as \eqref{eq:patch2patch}, and \eqref{M:semi bem}, one can easily verify that 
\begin{align}\label{eq:poincares}
h_T^{-1}\norm{v-v_T}{L^2(\pi_\coarse(T))} +h_T^{-1/2}\norm{(v-v_T)\widetilde\chi_T}{H^{1/2}(\Gamma)} +\norm{(v-v_T)\widetilde \chi_T}{H^1(\Gamma)} \lesssim \norm{ \nabla _\Gamma v}{L^2(\pi_\coarse(T))}.
\end{align}
We fix (independently of $\TT_\bullet$) a bounded domain $U\subset\R^d$ with $U_T\subset U$ for all $T\in\TT_\bullet$. 
Let   $\widetilde {\mathfrak{K}}:H^{1/2}(\Gamma)^D\to H^1(U\setminus \Gamma)^D$ denote the \emph{double-layer potential} from \cite[Theorem~6.11]{mclean}.
With these preparations and if $\mathfrak{P}$ has no lower-order terms, i.e., $b_i=0$ for all $i\in\{1,\dots,d\}$ as well as $c=0$, the proof of \eqref{eq:invest2} follows as in \cite[Section~5 and Section~6]{invest}.
In particular, one exploits the fact that 
\begin{align}\label{eq:K constant}
\widetilde{\mathfrak{K}} v_T=-v_T \text{ on } \Omega\quad\text{and}\quad \widetilde {\mathfrak{K}} v_T=0 \text{ on } \Omega^{\rm ext} :=\R^d\setminus\overline\Omega,
\end{align}
which follows from the representation formula \cite[Theorem~6.10]{mclean} together with the assumption that $\mathfrak{P}$ has no lower-order terms.
The proof of~\eqref{eq:invest2} then employs the Poincar\'e inequality \eqref{eq:poincares}, the property \eqref{eq:K constant}, the stability of ${\mathfrak{K}}:H^1(\Gamma)^D\to H^1(\Gamma)^D$ and of $\mathfrak{W}:H^1(\Gamma)^D\to L^2(\Gamma)^D$ from \cite[Corollary~3.38]{mitrea} or  \cite[page~209]{mclean}, and the Caccioppoli inequality~\eqref{eq:cacc} in combination with the transmission property \cite[Theorem~6.11]{mclean}.

To prove \eqref{eq:invest2} for general $\mathfrak{P}$ with lower-order terms, where \eqref{eq:K constant} is in general false, 
we recall the principal part $\mathfrak{P}_0$ from \eqref{eq:principal part}.
In the following five steps, we  show for the corresponding double-layer operator $\mathfrak{K}_0:H^{1/2}(\Gamma)\to H^{1/2}(\Gamma)$ and the hyper-singular operator $\mathfrak{W}_0:H^{-1/2}(\Gamma)\to H^{1/2}(\Gamma)$ that 
\begin{align}\label{eq:differences}
\mathfrak{K}-\mathfrak{K}_0:H^{1/2}(\Gamma)\to H^1(\Gamma) \quad\text{and}\quad\mathfrak{W}-\mathfrak{W}_0:H^{1/2}(\Gamma)\to L^2(\Gamma)
\end{align}
are continuous.
Since \eqref{eq:invest2} is satisfied for the operators corresponding to $\mathfrak{P}_0$, this and the trivial estimate $w_\coarse\lesssim \norm{w_\coarse/h_\coarse^{1/2}}{L^\infty(\Gamma)}$ will conclude the proof.

{\bf Step~1:}
Let $\widetilde{\mathfrak{N}},\widetilde{\mathfrak{N}}_0$ be the \emph{Newton potentials} from \cite[Theorem~6.1]{mclean} corresponding to $\mathfrak{P},\mathfrak{P}_0$.
According to \cite[Theorem~6.1]{mclean}, they satisfy the mapping property $\widetilde{\mathfrak{N}},\widetilde{\mathfrak{N}}_0:H^\sigma(\R^d)^D\to H^{\sigma+2}(\R^d)^D$ for all $\sigma\in\R$. 
In the proof of the latter stability, the fundamental solution is defined in terms of the Fourier transformation.
The definition involves a multivariate polynomial $P:\R^d\to\C^{D\times D}$ resp.\ $P_0:\R^d\to\C^{D\times D}$ associated to $\mathfrak{P}$ resp.\ $\mathfrak{P}_0$ (which is obtained from the differential operator by replacing the derivatives with variables) such that $|P(\xi)| \lesssim |P_0(\xi)^{-1}| = \OO(|\xi|^{-2})$ for $\xi\in\R^d$ and $|\xi|\to\infty$; see~\cite[Equation~(6.7)]{mclean}.
Indeed, the latter inequality is the key of the proof.
As elementary analysis  even shows that $|P(\xi)^{-1}-P_0(\xi)^{-1}|=|P(\xi)^{-1}[I-P(\xi)P_0(\xi)^{-1}]|=\OO(|\xi|^{-3}$) with the identity matrix $I\in\C^{D\times D}$, one sees along the lines of \cite[Theorem~6.1]{mclean} the additional regularity
\begin{align}\label{eq:newton diff}
\widetilde{\mathfrak{N}}-\widetilde{\mathfrak{N}}_0:H^\sigma(\R^d)^D\to H^{\sigma+3}(\R^d)^D.
\end{align}
Since multiplication with a fixed compactly supported smooth function is stable (see, e.g., \cite[Theorem~3.20]{mclean}), we also see for  arbitrary $\chi_1,\chi_2\in C_c^\infty(\R^d)$ the continuity of
\begin{align}\label{eq:newton diff2}
\mathfrak{A}:H^\sigma(\R^d)^D\to H^{\sigma+3}(\R^d)^D, \quad g \mapsto \chi_1(\widetilde{\mathfrak{N}}-\widetilde{\mathfrak{N}}_0)(g \chi_2).
\end{align}

{\bf Step~2:}
For sufficiently large $\lambda>0$, the sesquilinear forms $\sprod{\cdot}{\cdot}_{\mathfrak{P}+\lambda}$ and $\sprod{\cdot}{\cdot}_{\mathfrak{P}_0+\lambda}$ from~\eqref{eq:PDE form} are both even elliptic on $H_0^1(\Omega)^D$.
Let $u_\lambda\in H^1(\Omega)$ be the unique weak solution of $(\mathfrak{P}+\lambda)u_\lambda=0$ and $u_\lambda|_\Gamma=v$.
Similarly, let  $u_{0,\lambda}\in H^1(\Omega)$ be the solution of $(\mathfrak{P}_0+\lambda)u_{0,\lambda}=0$ and $u_{0,\lambda}|_\Gamma=v$.
We extend both functions by zero outside of $\Omega$.
Since $\mathfrak{P} u_\lambda=-u_\lambda$ and $\mathfrak{P}_0 u_{0,\lambda}=-u_{0,\lambda}$, the representation formula \cite[Theorem~6.10]{mclean} yields that 
\begin{align}\label{eq:K representation}
(\widetilde{\mathfrak{K}}-\widetilde{\mathfrak{K}}_0)v
=(u_\lambda-u_{0,\lambda})
-\lambda (\widetilde{\mathfrak{N}}u_\lambda-\widetilde{\mathfrak{N}}_0u_{0,\lambda} )
+(\widetilde {\mathfrak{V}}\mathfrak{D}_\nu u_\lambda-\widetilde{\mathfrak{V}}_0\mathfrak{D}_\nu u_{0,\lambda}).
\end{align}
We note that $\mathfrak{K}-\mathfrak{K}_0=(\widetilde{\mathfrak{K}}-\widetilde{\mathfrak{K}}_0)(\cdot)|_{\Gamma}$ and $\mathfrak{W}-\mathfrak{W}_0=-\mathfrak{D}_\nu (\widetilde{\mathfrak{K}}-\widetilde{\mathfrak{K}}_0)$ with the conormal derivative $\mathfrak{D}_\nu(\cdot)$.
To see~\eqref{eq:differences} and hence to conclude the proof, we thus only have to bound the trace as well as the conormal derivative of each summand in \eqref{eq:K representation} separately, which will be done in the following three steps.

{\bf Step~3:}
By definition, the trace of the first summand in \eqref{eq:K representation} vanishes, i.e., $(u_\lambda-u_{0,\lambda})|_\Gamma=0$.
According to \eqref{eq:regular normal} (where the differential operator can also be chosen as $\mathfrak{P}+\lambda$ instead of $\mathfrak{P}$),    the  normal derivative satisfies that
\begin{align*}
\norm{\mathfrak{D}_\nu(u_\lambda-u_{0,\lambda})}{L^2(\Gamma)}
\stackrel{\eqref{eq:regular normal}}\lesssim \norm{(u_\lambda-u_{0,\lambda})|_\Gamma}{H^1(\Gamma)} +\norm{u_\lambda-u_{0,\lambda}}{H^1(\Omega)}+\norm{(\mathfrak{P}+\lambda)(u_\lambda-u_{0,\lambda})}{L^2(\Omega)}.
\end{align*}
Since the first summand vanishes and $ \norm{(\mathfrak{P}+\lambda)(u_\lambda-u_{0,\lambda})}{L^2(\Omega)}\lesssim\norm{u_{0,\lambda}}{H^1(\Omega)}$ due to $(\mathfrak{P}+\lambda)(u_\lambda-u_{0,\lambda})=-\big(\sum_{i=1}^d b_i\partial_i u_{0,\lambda}\big) -c u_{0,\lambda}$, the stability of the solution mapping $v\mapsto u_\lambda$ and $v\mapsto u_{0,\lambda}$ gives that
\begin{align}
\label{eq:solution difference}
\norm{\mathfrak{D}_\nu(u_\lambda-u_{0,\lambda})}{L^2(\Gamma)}
\lesssim\norm{v}{H^{1/2}(\Gamma)}. 
\end{align}

{\bf Step~4:}
Due to our assumption that the trace operator $(\cdot)|_\Gamma:H^{3/2}(\Omega)^D\to H^1(\Gamma)^D$ is well-defined and continuous, the stability of the Newton potentials, the stability of the solution mapping $v\mapsto u_\lambda$ and $v\mapsto u_{0,\lambda}$, the trace of the second summand in \eqref{eq:K representation}   satisfies that  
\begin{align*}
&\norm{(\widetilde{\mathfrak{N}}u_\lambda-\widetilde{\mathfrak{N}}_0u_{0,\lambda})|_\Gamma}{H^1(\Gamma)}
\le  \norm{(\widetilde{\mathfrak{N}}u_\lambda)|_\Gamma}{H^{1}(\Gamma)}
+\norm{(\widetilde{\mathfrak{N}}_0 u_{0,\lambda})|_\Gamma}{H^{1}(\Gamma)}
\\
&\quad\lesssim \norm{\widetilde{\mathfrak{N}}u_\lambda}{H^{3/2}(\Omega)}
+\norm{\widetilde{\mathfrak{N}}_0 u_{0,\lambda}}{H^{3/2}(\Omega)}
\lesssim \norm{u_\lambda}{H^{-1/2}(\R^d)} +\norm{u_{0,\lambda}}{H^{-1/2}(\R^d)}
\\
&\quad\lesssim \norm{u_\lambda}{L^2(\R^d)} +\norm{u_{0,\lambda}}{L^2(\R^d)}
= \norm{u_\lambda}{L^2(\Omega)} +\norm{u_{0,\lambda}}{L^2(\Omega)}
\lesssim \norm{v}{H^{1/2}(\Gamma)}.
\end{align*}

Note that $\widetilde{\mathfrak{N}}$ and $\widetilde{\mathfrak{N}}_0$ are indeed potentials, i.e., $\mathfrak{P}\widetilde{\mathfrak{N}}=0$ weakly and $\mathfrak{P}_0\widetilde{\mathfrak{N}}_0=0$ weakly. 
With \eqref{eq:regular normal} (applied for $\mathfrak{P}$ and $\mathfrak{P}_0$), the stability of the Newton potentials, and the estimates for the trace of $\widetilde{\mathfrak{N}}u_\lambda$ and $\widetilde{\mathfrak{N}}_0 u_{0,\lambda}$, we thus see that 
\begin{align*}
&\norm{\mathfrak{D}_\nu (\widetilde{\mathfrak{N}}u_\lambda-\widetilde{\mathfrak{N}}_0u_{0,\lambda})}{L^2(\Gamma)}
\stackrel{\eqref{eq:regular normal}}\lesssim \norm{(\widetilde{\mathfrak{N}}u_\lambda)|_\Gamma}{H^1(\Gamma)}+\norm{\widetilde{\mathfrak{N}}u_\lambda}{H^1(\Omega)}
+\norm{(\widetilde{\mathfrak{N}}_0u_{0,\lambda})|_\Gamma}{H^1(\Gamma)}+\norm{\widetilde{\mathfrak{N}}_0 u_{0,\lambda}}{H^1(\Omega)}
\\
&\quad\lesssim  \norm{v}{H^{1/2}(\Gamma)} + \norm{u_\lambda}{H^{-1}(\R^d)}
+\norm{u_{0,\lambda}}{H^{-1}(\R^d)}
\lesssim \norm{v}{H^{1/2}(\Gamma)} + \norm{u_\lambda}{L^2(\R^d)}
+\norm{u_{0,\lambda}}{L^2(\R^d)}.
\end{align*}
Again, stability of the solution mappings allows to bound the last terms by $\norm{v}{H^{1/2}(\Gamma)}$.

{\bf Step~5:}
To deal with the third summand in \eqref{eq:K representation}, we first rewrite it as follows
\begin{align}\label{eq:V splitting}
\widetilde {\mathfrak{V}}\mathfrak{D}_\nu u_\lambda-\widetilde{\mathfrak{V}}_0\mathfrak{D}_\nu u_{0,\lambda}
=\widetilde {\mathfrak{V}}\mathfrak{D}_\nu (u_\lambda-u_{0,\lambda})
+(\widetilde{\mathfrak{V}}-\widetilde{\mathfrak{V}}_0)\mathfrak{D}_\nu u_{0,\lambda}.
\end{align}
Due to the stability of $\widetilde {\mathfrak{V}}(\cdot)|_\Gamma=\mathfrak{V}:L^2(\Gamma)\to H^1(\Gamma)$ (see \eqref{eq:single layer operator} and \eqref{eq:trace of V}) and \eqref{eq:solution difference}, we have for the first summand in \eqref{eq:V splitting} that  
\begin{align}\label{eq:Vtrace to v}
\norm{(\widetilde {\mathfrak{V}}\mathfrak{D}_\nu (u_\lambda-u_{0,\lambda}))|_\Gamma}{H^1(\Gamma)}\lesssim\norm{v}{H^{1/2}(\Gamma)}.
\end{align}
To deal with the conormal derivative, we apply again \eqref{eq:regular normal} together with the fact that $\widetilde {\mathfrak{V}}$ is a potential, i.e., $\mathfrak{P}\widetilde {\mathfrak{V}}=0$ weakly.
This leads to
\begin{align*}
\norm{\mathfrak{D}_\nu\widetilde {\mathfrak{V}}\mathfrak{D}_\nu (u_\lambda-u_{0,\lambda})}{L^2(\Gamma)}
\stackrel{\eqref{eq:regular normal}}\lesssim \norm{(\widetilde {\mathfrak{V}}\mathfrak{D}_\nu (u_\lambda-u_{0,\lambda}))|_\Gamma}{H^1(\Gamma)}+\norm{\widetilde {\mathfrak{V}}\mathfrak{D}_\nu (u_\lambda-u_{0,\lambda})}{H^1(\Omega)}.
\end{align*}
The first summand can be bounded as in \eqref{eq:Vtrace to v}.
For the second  summand, we use the stability~\eqref{eq:single layer potential} in combination with $\norm{\cdot}{H^{-1/2}(\Gamma)} \lesssim \norm{\cdot}{L^2(\Gamma)}$ and \eqref{eq:solution difference}.

Finally, it only remains to bound the trace as well as the conormal derivative of the second summand in \eqref{eq:V splitting}. 
Choosing $\sigma=-1$ in the additional regularity \eqref{eq:newton diff2}, one sees as in the proof of \cite[page~203]{mclean} (which proves stability of $\widetilde {\mathfrak{V}}:H^{-1/2}(\Gamma)^D\to H^1(\Omega)^D$) that  
\begin{align}\label{eq:V difference}
\widetilde{\mathfrak{V}}-\widetilde{\mathfrak{V}}_0:H^{-1/2}(\Gamma)^D\to H^2(\Omega)^D
\end{align}
With the assumption $(\cdot)|_\Gamma:H^{3/2}(\Omega)^D\to H^1(\Gamma)^D$, this implies that
\begin{align}\label{eq:last estimate}
\norm{((\widetilde{\mathfrak{V}}-\widetilde{\mathfrak{V}}_0)\mathfrak{D}_\nu u_{0,\lambda})|_\Gamma}{H^1(\Gamma)}
\lesssim\norm{(\widetilde{\mathfrak{V}}-\widetilde{\mathfrak{V}}_0)\mathfrak{D}_\nu u_{0,\lambda}}{H^{3/2}(\Omega)}
\lesssim \norm{\mathfrak{D}_\nu u_{0,\lambda}}{H^{-1/2}(\Gamma)}.
\end{align}
Recall that $(\mathfrak{P}_0+\lambda) u_{0,\lambda}=0$ with $u_{0,\lambda}|_\Gamma=v$. Thus, \cite[Theorem~4.25]{mclean} (which states boundedness of the Dirichlet to Neumann mapping) gives that
\begin{align}\label{eq:trace V diff}
\norm{((\widetilde{\mathfrak{V}}-\widetilde{\mathfrak{V}}_0)\mathfrak{D}_\nu u_{0,\lambda})|_\Gamma}{H^1(\Gamma)}
\stackrel{\eqref{eq:last estimate}}\lesssim \norm{\mathfrak{D}_\nu u_{0,\lambda}}{H^{-1/2}(\Gamma)}\lesssim\norm{v}{H^{1/2}(\Gamma)}.
\end{align}
Moreover, \eqref{eq:regular normal}, the fact that $\mathfrak{P} (\widetilde {\mathfrak{V}}-\widetilde{\mathfrak{V}}_0)=-\big(\sum_{i=1}^d b_i\partial_i \widetilde{\mathfrak{V}}_0\big) -c \widetilde{\mathfrak{V}}_0$, the stability \eqref{eq:single layer potential}, and \eqref{eq:trace V diff} show for the conormal derivative that
\begin{align*}
\norm{\mathfrak{D}_\nu(\widetilde{\mathfrak{V}}-\widetilde{\mathfrak{V}}_0)\mathfrak{D}_\nu u_{0,\lambda}}{L^2(\Gamma)}&\lesssim\norm{((\widetilde{\mathfrak{V}}-\widetilde{\mathfrak{V}}_0)\mathfrak{D}_\nu u_{0,\lambda})|_\Gamma}{H^1(\Gamma)}
+\norm{(\widetilde{\mathfrak{V}}-\widetilde{\mathfrak{V}}_0)\mathfrak{D}_\nu u_{0,\lambda}}{H^1(\Omega)}\\
&\quad+\norm{\mathfrak{P}(\widetilde{\mathfrak{V}}-\widetilde{\mathfrak{V}}_0)\mathfrak{D}_\nu u_{0,\lambda}}{L^2(\Omega)}\lesssim\norm{v}{H^{1/2}(\Gamma)}.
\end{align*} 
Overall, we have estimated the trace as well as the conormal derivative of all terms in \eqref{eq:K representation}.
Since $\mathfrak{K}-\mathfrak{K}_0=(\widetilde{\mathfrak{K}}-\widetilde{\mathfrak{K}}_0)(\cdot)|_{\Gamma}$ and $\mathfrak{W}-\mathfrak{W}_0=-\mathfrak{D}_\nu (\widetilde{\mathfrak{K}}-\widetilde{\mathfrak{K}}_0)$, this verifies the stability~\eqref{eq:differences} and thus concludes the proof.
\end{proof}

\section*{Acknowledgement} The authors acknowledge support through the Austrian Science Fund (FWF) under grant P29096, grant W1245, and J4379-N.

\bibliographystyle{alpha}
\bibliography{literature}

\end{document}